\documentclass{amsart}
\usepackage{amssymb,mathtools,mathrsfs,booktabs,graphicx}
\usepackage{array,multirow}
\usepackage{microtype}
\allowdisplaybreaks[2]
\usepackage{placeins}
\usepackage{needspace}
\usepackage{float}
\usepackage{epstopdf}
\usepackage{aliascnt}
\usepackage{hyperref}
\usepackage[capitalize,nameinlink]{cleveref}
\newcommand{\R}{\mathbb R}
\newcommand{\C}{\mathbb C}

\newcommand{\ip}[1]{\langle#1\rangle}
\newcommand{\Sph}{\mathbb S^{d-1}}
\newcommand{\eps}{\varepsilon}
\newcommand{\dd}{\,\mathrm d}
\newcommand{\ii}{\mathrm i}
\newcommand{\norm}[1]{\lVert#1\rVert}
\DeclareMathOperator{\supp}{supp}

\DeclareMathOperator{\dist}{dist}
\DeclareMathOperator{\sech}{sech}
\DeclareMathOperator{\spanop}{span}
\theoremstyle{plain}
\newtheorem{theorem}{Theorem}[section]
\newaliascnt{lemma}{theorem}
\newtheorem{lemma}[lemma]{Lemma}
\aliascntresetthe{lemma}
\newaliascnt{proposition}{theorem}
\newtheorem{proposition}[proposition]{Proposition}
\aliascntresetthe{proposition}
\newaliascnt{corollary}{theorem}
\newtheorem{corollary}[corollary]{Corollary}
\aliascntresetthe{corollary}
\theoremstyle{definition}
\newaliascnt{definition}{theorem}

\aliascntresetthe{definition}
\theoremstyle{remark}
\newaliascnt{remark}{theorem}
\newtheorem{remark}[remark]{Remark}
\aliascntresetthe{remark}
\numberwithin{equation}{section}

\title[Deterministic Shallow Networks]{Fully Deterministic Shallow Neural Networks with Localized Activations: Construction and Optimal Sobolev Approximation}
\author{Ran Bi}
\address{School of Mathematics, Nanjing University, Nanjing 210093, People's Republic of China}
\email{ranbi@smail.nju.edu.cn}
\author{Weibing Deng}
\address{School of Mathematics, Nanjing University, Nanjing 210093, People's Republic of China}
\email{wbdeng@nju.edu.cn}
\thanks{This work was supported by the National Key R\&D Program of China (2024YFA1012600) and the NSF of China (12171237).}
\date{}
\hypersetup{colorlinks=true,allcolors=black,pdftitle={Fully Deterministic Shallow Neural Networks with Localized Activations: Construction and Optimal Sobolev Approximation},pdfauthor={Ran Bi and Weibing Deng}}
\begin{document}
\begin{abstract}
We construct fully deterministic finite-dimensional spaces for Sobolev
approximation using shallow neural networks with a localized activation
$\sigma$, either Gaussian or tanh-difference. For a bounded Lipschitz
domain $\Omega\subset\mathbb R^d$, $d\ge2$, and a prescribed regularity
index $s>0$, the approximation space at frequency resolution $M$ is
\[
 V_M^\Omega=\spanop\left\{
 \sigma\!\left(\frac{\omega_i\cdot x-b_{j,M}}{\eps_M}\right):
 1\le i\le J_M,\quad 1\le j\le K_{i,M}
 \right\}.
\]
All parameters are prescribed explicitly and
independently of the target function. The offset cutoff depends on
the direction and the domain, with an exterior margin tending to zero.
The feature count
satisfies $N_M=\sum_{i=1}^{J_M}K_{i,M}\asymp M^d$.
For sufficiently large $M$, we construct a bounded linear operator
$\mathcal A_M^\Omega:H^s(\Omega)\to V_M^\Omega$ that achieves the optimal approximation rate
\[
 \|u-\mathcal A_M^\Omega u\|_{H^m(\Omega)}
 \le C N_M^{-(s-m)/d}\|u\|_{H^s(\Omega)}
\]
for every $u\in H^s(\Omega)$ and every integer $0\le m<s$.
Numerical experiments
in two and three dimensions illustrate the approximation behavior in
$L^2$ and $H^1$.
\end{abstract}
\maketitle
\section{Introduction}\label{sec:introduction}

Numerical methods for partial differential equations often begin with
the choice of a finite-dimensional approximation space. With prescribed
hidden parameters, a shallow neural network has the form of a linear
combination of fixed features. Their span defines such an approximation
space, and the output coefficients can be determined by least-squares
or variational formulations. Examples of this approach include physics-informed extreme
learning machines \cite{DwivediSrinivasan2020}, local extreme learning
machines \cite{DongLi2021}, and the random feature method
\cite{ChenChiEYang2022}, with extensions to interface problems
\cite{ChiChenYang2024} and discontinuous Galerkin formulations
\cite{SunDongWang2024}. The approximation properties of these spaces
depend on the choice of hidden parameters~\cite{ShenYangZhang2022Intrinsic}.

Classical universal approximation theorems
\cite{Cybenko1989,LeshnoLinPinkusSchocken1993} show that shallow neural
networks can approximate any continuous function uniformly on compact
sets under suitable assumptions on the activation function.
Quantitative estimates relate the approximation error to the number of
neurons, the target class, and the error norm. Barron \cite{Barron1993}
obtained an $L^2$ error bound of order $N^{-1/2}$ under a finite first
absolute Fourier moment condition. Subsequent work developed
probabilistic constructions and coefficient bounds
\cite{Makovoz1996,KlusowskiBarron2018}, as well as Barron spaces defined
through integral representations and estimates for regularized learning
\cite{EMaWu2019,EMaWu2022}. Spectral estimates have also been obtained
for general activations, including higher-order rates for cosine and
$\mathrm{ReLU}^k$ expansions \cite{SiegelXu2020,SiegelXu2022HighOrder}.

Shen et al.~\cite{ShenYangZhang2022WidthDepth} bounded the
$L^p$ approximation error for H\"older continuous functions in terms of
the width and depth of a ReLU network. For smooth targets, related
estimates cover networks with encodable weights~\cite{GuhringRaslan2021},
deep ReLU networks~\cite{LuShenYangZhang2021}, and tanh networks with
two hidden layers~\cite{DeRyckLanthalerMishra2021}.
For shallow networks, related results include sharp variation-space
bounds \cite{SiegelXu2024Sharp} and $L^p$ approximation of Sobolev targets
using target-dependent $\mathrm{ReLU}^k$ parameters
\cite{MaoSiegelXu2026Radon}.

Random feature approximation connects parameter sampling with kernel
quadrature \cite{Bach2017Quadrature}. Approximation estimates have been
established for Black-Scholes-type PDEs \cite{Gonon2023} and
Banach-space-valued models \cite{NeufeldSchmocker2023}. Recent results
include randomized Sobolev
approximation \cite{DeRyckMishraShangWang2025}, spectral convergence of
random Fourier features \cite{MingYu2026Spectral}, and uniformly sampled
localized features in generalized spectral Barron spaces
\cite{BiDeng2026Adaptive}. Adaptive methods adjust the sampling
distribution \cite{YangWang2026AdaptiveDistribution} or a partition of
unity \cite{BiDeng2026Adaptive}. These approximation guarantees depend
on the sampling distribution and the target class.

For deterministic spaces with prescribed hidden parameters, Mhaskar
\cite{Mhaskar1996} constructed approximants with continuous linear
coefficient functionals under suitable activation assumptions, and
Petrushev \cite{Petrushev1998} combined prescribed ridge directions with
univariate approximation spaces. Liu et~al.\ \cite{LiuMaoXu2025Relu}
obtained optimal Sobolev rates for prescribed $\mathrm{ReLU}^k$ parameters
within specified smoothness ranges. For smooth activations, Li et~al.\
\cite{LiMaoXu2026Analytic} obtain multivariate Sobolev estimates from
tensor products of quasi-Chebyshev parameter sets under quantitative
lower bounds on the magnitudes of all Taylor coefficients. Their separate
construction for tanh yields one-dimensional Sobolev estimates and
multivariate $L^2$ approximation through ridge lifting.
Fu and Wang \cite{fu2026optimal} establish optimal Sobolev rates for
smooth sigmoidal activations, with directions chosen from spherical
designs and offsets in a fixed enlarged interval. In dimensions
$d\ge3$, their directional construction uses existence results for
such designs.
This raises the question of how to construct fully explicit spaces with
localized activations that attain optimal Sobolev rates while adapting
the offset cutoffs to the domain.

In this paper, we construct fully deterministic finite-dimensional
approximation spaces on a nonempty bounded Lipschitz domain
$\Omega\subset\mathbb R^d$, $d\ge2$. For a fixed regularity index
$s>0$ and an activation $\sigma$, we define the space at frequency
resolution $M$ by
\begin{equation*}
 V_M^\Omega=\spanop\left\{
 \sigma\!\left(\frac{\omega_i\cdot x-b_{j,M}}{\eps_M}\right):
 1\le i\le J_M,\quad 1\le j\le K_{i,M}
 \right\}.
\end{equation*}
We prescribe the directions
$\omega_i\in\mathbb S^{d-1}$, offsets $b_{j,M}$, and scale $\eps_M>0$
independently of the target function and derivative order.
For each direction $\omega_i$, we retain $K_{i,M}$ offsets according
to the geometry of $\Omega$. The resulting space is generated by
$N_M$ features, with
\[
 N_M=\sum_{i=1}^{J_M}K_{i,M}\asymp M^d,
 \qquad \dim V_M^\Omega\le N_M.
\]

For the spaces $V_M^\Omega$, we use activations that decay at both
ends of the real axis, so each ridge feature is localized in its
normal direction. This choice is also motivated by conditioning:
sigmoidal features that are nearly constant on $\Omega$ can be nearly
linearly dependent. The Gaussian $\sigma_G(t)=e^{-t^2}$ is a natural
smooth, rapidly decaying choice. Decaying activations can also be
constructed by subtracting two fixed translates of a sigmoidal
function, which cancels their common limiting values \cite{SiegelXu2020}.
Applying this construction to tanh gives
$\sigma_D(t)=\tanh(t+1)-\tanh(t-1)$. Each difference feature has one
free output coefficient, since the coefficients of its two tanh units
are constrained to be negatives of each other.

The Fourier transform of the Gaussian is nonvanishing, whereas that
of the tanh difference has real zeros. These zeros prevent direct
division by the activation's Fourier transform when constructing a
ridge representation. The analysis of Fu and Wang \cite{fu2026optimal}
uses sigmoidal activations whose derivatives have nonvanishing Fourier
transforms. For the tanh difference, we exclude suitable neighborhoods
of the zeros and control the resulting Sobolev error. This yields an
exact ridge representation of a modified target at the prescribed
scale and preserves the optimal approximation rate for this activation.
Extending the argument to other sigmoidal differences requires
corresponding smoothness, decay, and Fourier bounds.

For these spaces, we construct a bounded linear approximation operator
\mbox{$\mathcal A_M^\Omega:H^s(\Omega)\to V_M^\Omega$} satisfying
\begin{equation}\label{intro:rate}
 \norm{u-\mathcal A_M^\Omega u}_{H^m(\Omega)} \le C N_M^{-(s-m)/d}\norm u_{H^s(\Omega)},
 \qquad m\in\mathbb N_0,\quad m<s,
\end{equation}
for every $u\in H^s(\Omega)$ and all sufficiently large $M$.
Classical Sobolev width estimates
\cite{DahlkeNovakSickel2006II} give a matching lower bound for
fixed linear spaces of dimension at most $N_M$. Thus the rate in
\eqref{intro:rate} is optimal in this class.

Our main contributions are as follows.
\begin{enumerate}
\item \emph{Explicit deterministic spaces with optimal Sobolev approximation.}
For both activations and every fixed $s>0$, we construct spaces on
bounded Lipschitz domains in $\mathbb R^d$, $d\ge2$, with
$N_M\asymp M^d$ prescribed features. The directions, offsets, and
scales are specified by explicit deterministic rules independently of
the target function and derivative order. The spaces admit bounded
linear approximation operators that attain the optimal rate for fixed
linear spaces. 

\item \emph{Optimal approximation with real Fourier zeros.}
We establish the optimal Sobolev rate for the tanh difference despite
the infinitely many real zeros of its Fourier transform. By excluding
suitable neighborhoods of these zeros, we construct an exact ridge
representation of a modified target at the prescribed scale. The
exclusion error has the required Sobolev order.

\item \emph{Offset truncation adapted to the domain.}
The offset cutoff is defined using the support function of the domain, with
exterior margins of order $M^{-1}\log M$ for the Gaussian and
$M^{-1}(\log M)^2$ for the tanh difference. An $L^2$ estimate for
weighted exponential sums controls the discarded contributions without
any smoothness assumption on the support function. For the tanh
difference, correcting the coefficient at the last retained node in
each direction cancels the leading term of the discarded tail without
adding features. These estimates preserve the optimal Sobolev order
as the exterior margins tend to zero.

\item \emph{Coefficient bounds and regularized approximation.}
The sum of the absolute values of the constructed coefficients grows
at most polynomially in $M$.
For the original features and those normalized in the continuous
$H^m(\Omega)$ norm, this bound yields an explicit sufficient
condition on the relative threshold for truncated singular value
decomposition (TSVD) to preserve the optimal approximation order.
\end{enumerate}

Section~\ref{sec:construction} specifies the approximation spaces and states the Sobolev
approximation theorem with coefficient bounds.
Section~\ref{sec:representation} constructs the continuous
ridge representation. Section~\ref{sec:discretization} derives coefficient
bounds and discretizes the ridge integrals to prove the approximation theorem.
Section~\ref{sec:regularized-ls} establishes the continuous TSVD
approximation estimates.
Section~\ref{sec:numerics} presents numerical experiments in two and
three dimensions with Gaussian and tanh-difference activations.

Throughout the paper, $c$, $C$, and $C'$ denote generic positive constants
whose values can change from one estimate to the next. Unless otherwise
stated, the constants in the approximation estimates depend only on the
dimension $d$, the domain $\Omega$, the regularity index $s$, and the fixed
choices in the construction. They are independent of the frequency
resolution $M$ and the target function $u$, and are uniform over the
integers $0\le m<s$. For positive quantities $A$ and $B$, we write
$A\asymp B$ if $cB\le A\le CB$ for positive constants $c$ and $C$
independent of the parameters varying in the comparison.

\section{Network construction and approximation theorem}\label{sec:construction}

\subsection{Notation and preliminaries}
Let $\Omega\subset\R^d$ be a nonempty bounded Lipschitz domain, $d\ge2$.
Choose an interior point as the origin by translating the coordinates,
and retain the notation $\Omega$ for the translated domain.
Write $B(x,r)$ for the open Euclidean ball centered at $x$ with radius
$r$, and set $B_r=B(0,r)$. Define
\mbox{$r_\Omega=\dist(0,\partial\Omega)$} and
\mbox{$R_\Omega=\sup_{x\in\Omega}|x|$}, where $|x|$ denotes the Euclidean norm.
Since $\Omega$ is open and bounded, $0<r_\Omega\le R_\Omega<\infty$
and $B_{r_\Omega}\subset\Omega\subset B_{R_\Omega}$.

For a unit direction $\omega\in\Sph$, define the support function
$H_\Omega(\omega)=\sup_{x\in\Omega}\omega\cdot x$. The ball inclusions
and this definition imply
\[
 r_\Omega\le H_\Omega(\omega)\le R_\Omega,\qquad
 -H_\Omega(-\omega)\le\omega\cdot x\le H_\Omega(\omega)
 \quad(x\in\Omega).
\]

Fix a Sobolev regularity index $s>0$. The derivative orders considered
below are the integers $0\le m<s$, whose largest value is
$m_*:=\lceil s\rceil-1$. For a multi-index $\alpha$, $|\alpha|$ denotes
its total order. For integrable functions, we use the Fourier transform
$\widehat f(\xi)=\int_{\R^d}f(x)e^{-\ii x\cdot\xi}\dd x$, with inverse
normalization $(2\pi)^{-d}$ and the usual extension to $L^2(\R^d)$.
Write $\ip v=(1+|v|^2)^{1/2}$ for a scalar or vector $v$.
The Bessel potential space $H^s(\R^d)$ is equipped with the norm
\[
 \norm f_{H^s(\R^d)}^2
 =(2\pi)^{-d}\int_{\R^d}\ip\xi^{2s}|\widehat f(\xi)|^2\dd\xi.
\]
We define $H^s(\Omega)$ by restriction from $H^s(\R^d)$ and equip it
with the quotient norm. At integer orders $0\le m\le m_*$, we use the
equivalent full weak-derivative norm \cite{Brezis2011}
\[
 \norm u_{H^m(\Omega)}^2
 =\sum_{|\alpha|\le m}\norm{\partial^\alpha u}_{L^2(\Omega)}^2.
\]

\subsection{Construction of the approximation space}\label{subsec:space-construction}
Fix an activation $\sigma\in\{\sigma_G,\sigma_D\}$, where
\begin{equation*}
 \sigma_G(t)=e^{-t^2},\qquad
 \sigma_D(t)=\tanh(t+1)-\tanh(t-1).
\end{equation*}
At frequency resolution $M>0$, consider directions
$\omega_1,\ldots,\omega_{J_M}\in\Sph$ and positive offsets
$b_{\sigma,j,M}$ with $1\le j\le K_{\sigma,i,M}$ in each direction
$\omega_i$. For a scale $\eps_{\sigma,M}>0$, define the ridge features
by
\begin{equation}\label{eq:ridge-features}
 \phi_{\sigma,ij,M}(x)
 =\sigma\!\left(\frac{\omega_i\cdot x-b_{\sigma,j,M}}{\eps_{\sigma,M}}\right).
\end{equation}
These features span the finite-dimensional linear space
\begin{equation}\label{eq:space}
 V_{\sigma,M}^\Omega
 =\spanop\{\phi_{\sigma,ij,M}:1\le i\le J_M,\ 1\le j\le K_{\sigma,i,M}\}.
\end{equation}
The span is taken over the scalar field of $H^s(\Omega)$.
Each feature is localized near the hyperplane
$\omega_i\cdot x=b_{\sigma,j,M}$ on the scale $\eps_{\sigma,M}$ in the
normal direction. Its hidden weight and bias are
$\omega_i/\eps_{\sigma,M}$ and $-b_{\sigma,j,M}/\eps_{\sigma,M}$.
The following rules prescribe these parameters independently of the
target function $u$ and the derivative order $m$.

Set $T_M=\log(M+1)>0$ and choose suitable positive constants $a_G$ and $a_D$.
Define the feature scale and offset margin by
\begin{equation}\label{eq:parameters}
 \renewcommand{\arraystretch}{1.2}
 \begin{array}{c@{\quad}|cc}
 \toprule
 &\sigma_G&\sigma_D\\
 \midrule
 \eps_{\sigma,M}\;\text{(feature scale)}&\dfrac{\sqrt{T_M}}M&\dfrac{T_M}M\\[3pt]
 \Delta_{\sigma,M}\;\text{(offset margin)}&\dfrac{a_GT_M}M&\dfrac{a_DT_M^2}M\\
 \bottomrule
 \end{array}.
\end{equation}
To construct the directions, choose an integer $L_M\asymp M$.
The point sets $\Theta_M^{(k)}\subset\mathbb S^{k-1}$ are defined
recursively, starting with the equally spaced circle points
\[
 \Theta_M^{(2)}
 =\{(\cos(\pi\ell/L_M),\sin(\pi\ell/L_M)):0\le\ell<2L_M\}.
\]
For $k=3,\ldots,d$, let $t_1^{(k)},\ldots,t_{L_M}^{(k)}$ be the
$L_M$-point Gauss--Jacobi nodes for the weight
$(1-t^2)^{(k-3)/2}$ on $(-1,1)$.
Taking each node as the last coordinate and rescaling the preceding
point set gives
\[
 \Theta_M^{(k)}
 =\left\{\left(\sqrt{1-(t_\ell^{(k)})^2}\,\zeta,t_\ell^{(k)}\right):
       \zeta\in\Theta_M^{(k-1)},\quad 1\le\ell\le L_M\right\}.
\]
We take the elements of $\Theta_M^{(d)}$ as the directions
$\omega_1,\ldots,\omega_{J_M}$. This set is antipodal and contains
$J_M\asymp M^{d-1}$ points. For $d=2$, it is the initial circle set.

For the offsets, we take midpoint nodes on the positive half-axis.
Section~\ref{subsec:offset-discretization} derives a spacing
$h_{\sigma,M}\asymp M^{-1}$ from the midpoint error estimate.
In each direction
$\omega_i$, retain the nodes up to
$H_\Omega(\omega_i)+\Delta_{\sigma,M}$. Thus
\begin{equation}\label{eq:offset-nodes}
 b_{\sigma,j,M}=(j-\tfrac12)h_{\sigma,M},\qquad j\in\mathbb N,\qquad
 K_{\sigma,i,M}=\left\lfloor
 \frac{H_\Omega(\omega_i)+\Delta_{\sigma,M}}{h_{\sigma,M}}+\frac12 \right\rfloor.
\end{equation}
The retained indices are $1\le j\le K_{\sigma,i,M}$, or equivalently,
$b_{\sigma,j,M}\le H_\Omega(\omega_i)+\Delta_{\sigma,M}$.
The exterior margin tends to zero as $M\to\infty$.

The retained features are indexed by
\[
 I_{\sigma,M}^\Omega
 =\{(i,j):1\le i\le J_M,\ 1\le j\le K_{\sigma,i,M}\}.
\]
Since $H_\Omega$ is bounded above and below by positive constants on
$\Sph$ and $\Delta_{\sigma,M}\to0$, the number of offsets in each
direction is comparable to $h_{\sigma,M}^{-1}$, uniformly in $i$.
Combining this estimate with the direction count gives
\begin{equation*}
 K_{\sigma,i,M}\asymp M,\qquad N_{\sigma,M}:=\sum_iK_{\sigma,i,M}\asymp M^d,\qquad \dim
 V_{\sigma,M}^\Omega\le N_{\sigma,M}.
\end{equation*}
When an activation is fixed, we omit its subscript.

\subsection{Approximation theorem}
\begin{theorem}\label{thm:main}
Let $V_{\sigma,M}^\Omega$ be the space constructed in Section~2.2.
For all sufficiently large $M$, there is a bounded linear operator
$\mathcal A_{\sigma,M}^\Omega:H^s(\Omega)\to V_{\sigma,M}^\Omega$ satisfying
\begin{equation}\label{eq:main-error}
 \norm{u-\mathcal A_{\sigma,M}^\Omega u}_{H^m(\Omega)} \le CM^{m-s}\norm u_{H^s(\Omega)}
 \le C'N_{\sigma,M}^{-(s-m)/d}\norm u_{H^s(\Omega)}
\end{equation}
for every $u\in H^s(\Omega)$ and every integer $0\le m\le m_*$.
The operator has the form
\[
 \mathcal A_{\sigma,M}^\Omega u
 =\sum_{(i,j)\in I_{\sigma,M}^\Omega}
       c_{\sigma,ij,M}(u)\phi_{\sigma,ij,M},
\]
and its coefficients satisfy
\begin{equation}\label{eq:main-budget}
 \sum_{(i,j)\in I_{\sigma,M}^\Omega}|c_{\sigma,ij,M}(u)|
 \le CM^{P_\sigma}\norm u_{H^s(\Omega)},
\end{equation}
where
\begin{equation}\label{eq:coefficient-exponents}
 P_G=\max\Bigl\{1,\frac d2+\frac54-s\Bigr\},\qquad P_D=\max\Bigl\{1,\frac
 d2+\frac32+\frac\pi2-s\Bigr\}.
\end{equation}
\end{theorem}

\begin{remark}\label{rem:optimality}
Classical Sobolev width estimates \cite{DahlkeNovakSickel2006II}
imply that, for every integer $0\le m<s$ and every linear subspace
$V\subset H^m(\Omega)$ with $\dim V\le N$,
\begin{equation*}
 \sup_{\norm u_{H^s(\Omega)}\le1}\inf_{v\in V}
 \norm{u-v}_{H^m(\Omega)}
 \ge c(N+1)^{-(s-m)/d}.
\end{equation*}
Here $c>0$ is independent of $N$ and $V$.
Thus the rate in \cref{thm:main} is optimal in the worst case over the
Sobolev unit ball among linear spaces of dimension at most $N$ chosen
independently of the target.
\end{remark}

\section{Continuous ridge representation}\label{sec:representation}

We approximate $u\in H^s(\Omega)$ by truncating the Fourier transform
of a compactly supported extension. For the tanh difference, we also
exclude neighborhoods of the zeros of the activation's Fourier
transform. We estimate the resulting Sobolev error and construct an
exact ridge representation of the modified function.

Choose an auxiliary radius $R_0>R_\Omega$.
By the Lipschitz extension theorem \cite{Rychkov1999}
and a fixed smooth spatial cutoff, there is a bounded linear extension $E$
such that, for every $u\in H^s(\Omega)$, the function $f=Eu$ satisfies
\begin{equation}\label{eq:extension}
 f|_\Omega=u,\qquad \supp f\subset B_{R_0},\qquad \norm f_{H^s(\R^d)}\le C\norm
 u_{H^s(\Omega)}.
\end{equation}
In particular, $\norm f_{L^1}\le C\norm f_{L^2}\le C\norm u_{H^s(\Omega)}$.
We fix $E$ throughout the paper and extend it complex linearly.

\subsection{Frequency cutoff and spectral exclusions}
Let $\mathbf 1_A$ denote the indicator of a measurable set $A$.
Apply the sharp frequency cutoff
$\widehat f_M(\xi)=\mathbf 1_{\{|\xi|\le M\}}\widehat f(\xi)$.
Cauchy--Schwarz gives
$\norm{\widehat f_M}_{L^1(\R^d)}
\le CM^{d/2}\norm{\widehat f}_{L^2(\R^d)}$.
We therefore define $f_M$ by the absolutely convergent inverse Fourier
integral
\[
 f_M(x)=(2\pi)^{-d}\int_{|\xi|\le M}
             \widehat f(\xi)e^{\ii x\cdot\xi}\dd\xi.
\]
For every real $0\le t\le s$, the inequality
$\ip\xi^{t-s}\le M^{t-s}$ on $|\xi|>M$ and Plancherel's identity give
\begin{equation}\label{eq:hard-cutoff-error}
 \norm{f-f_M}_{H^t(\R^d)}
 \le CM^{t-s}\norm f_{H^s(\R^d)}.
\end{equation}
Restricting to $\Omega$ and using \eqref{eq:extension}, we obtain
\begin{equation}\label{eq:cutoff-domain-error}
 \norm{u-f_M}_{H^m(\Omega)}
 \le CM^{m-s}\norm u_{H^s(\Omega)},\qquad 0\le m\le m_*.
\end{equation}

Let $F_\sigma(\eta)=\int_\R\sigma(t)e^{-\ii\eta t}\dd t$ denote the
one-dimensional Fourier transform of the activation. For the two
activations under consideration,
\begin{equation}\label{eq:activation-transforms}
 F_G(\eta)=\sqrt\pi e^{-\eta^2/4},\qquad F_D(\eta)=\frac{2\pi\sin\eta}{\sinh(\pi\eta/2)},
 \qquad F_D(0)=4,
\end{equation}
where $F_D(0)$ is defined by continuity.
To derive the second identity, write $\sigma_D(t)=\int_{-1}^1\sech^2(t+v)\dd v$
and use the Fourier transform identity
\[
 \widehat{\sech^2}(\eta)
 =2\Gamma(1-\ii\eta/2)\Gamma(1+\ii\eta/2)
 =\frac{\pi\eta}{\sinh(\pi\eta/2)},
\]
which follows from the beta integral and the gamma reflection formula
\cite{OlverLozierBoisvertClark2010}.

Thus \(F_G\) has no real zeros, whereas the real zeros of \(F_D\) are exactly
\(k\pi\), \(k\in\mathbb Z\setminus\{0\}\).
Let \(\tau\) denote the signed radial frequency and let
\(\eta=\eps_M\tau\) be the scaled frequency.
For the Gaussian, no spectral exclusion is needed.
For the tanh difference, we take \(\eps_M=T_M/M\), so that the band
\(|\tau|\le M\) is mapped to \(|\eta|\le T_M\); within this scaled band
we exclude small neighborhoods of the relevant zeros of \(F_D\) as follows.

For $\sigma_D$, no exclusion is required on $[-\pi/2,\pi/2]$, on which
$F_D$ is nonzero. Around each positive zero, define the interval
\[
 \mathcal I_k=[k\pi-\pi/2,k\pi+\pi/2],\qquad k=1,2,\ldots.
\]
For sufficiently large $M$, set
\(
 K_M^{\rm root}=\left\lfloor\frac{T_M}{\pi}+\frac12\right\rfloor,
\)
so that $T_M\in \mathcal I_{K_M^{\rm root}}$ and the first $K_M^{\rm root}$
intervals cover $[\pi/2,T_M]$. We include the last interval even when its
center lies beyond the retained frequency band.
For $1\le k\le K_M^{\rm root}$, the scaled zero $k\pi$ corresponds
to the radial frequency
\(
 \rho_{k,M}=\frac{k\pi}{\eps_M}.
\)
To balance the exclusion error against the Sobolev frequency weights,
fix \mbox{$c_s=2^{-(2s+1)}$} and prescribe the radial half-width
\(
 \alpha_{k,M}=c_s\left(\frac{\rho_{k,M}}M\right)^{2s}.
\)
The change of variables $\eta=\eps_M\tau$ maps this radial interval
to an interval about $k\pi$ with half-width
\begin{equation}\label{eq:scaled-notch-widths}
 \delta_{k,M}=\eps_M\alpha_{k,M}
 =c_sM^{-1}T_M^{1-2s}(k\pi)^{2s}.
\end{equation}
Since $K_M^{\rm root}\pi\le T_M+\pi/2$, for sufficiently large $M$,
\begin{equation*}
 \rho_{k,M}\le M+\frac{\pi M}{2T_M}\le2M,\qquad \alpha_{k,M}\le\frac12,\qquad
 \delta_{k,M}\le\frac{\eps_M}2<\frac12.
\end{equation*}
The excluded intervals are therefore disjoint.

Excluding the corresponding neighborhoods of both positive and
negative zeros gives
\[
 \mathcal Z_M=\bigcup_{k=1}^{K_M^{\rm root}}\Bigl( [k\pi-\delta_{k,M},k\pi+\delta_{k,M}]
 \cup[-k\pi-\delta_{k,M},-k\pi+\delta_{k,M}]\Bigr).
\]

Define the spectral exclusion multiplier by
\begin{equation*}
 \vartheta_M(\eta)=
 \begin{cases}
  1,&\sigma=\sigma_G,\\[3pt]
  \mathbf 1_{\R\setminus\mathcal Z_M}(\eta),&\sigma=\sigma_D.
 \end{cases}
\end{equation*}
This multiplier is even and takes values in $\{0,1\}$.
Combining it with the frequency cutoff defines the modified function
$f_M^\star$ through
\begin{equation*}
 \widehat f_M^\star(\xi)=\mathbf 1_{\{|\xi|\le M\}}
       \vartheta_M(\eps_M|\xi|)\widehat f(\xi).
\end{equation*}
In particular, $f_M^\star=f_M$ for the Gaussian activation.

\begin{lemma}\label{lem:frequency}
For all sufficiently large $M$ and every integer $0\le m\le m_*$,
\begin{equation*}
 \norm{u-f_M^\star}_{H^m(\Omega)}
 \le CM^{m-s}\norm u_{H^s(\Omega)}.
\end{equation*}
\end{lemma}
\begin{proof}
For $\sigma_G$, $f_M^\star=f_M$, so \eqref{eq:cutoff-domain-error}
proves the assertion. For $\sigma_D$, compact support and boundedness
of smooth multiplication on $H^s$ give
$\norm{x_jf}_{H^s(\R^d)}\le C\norm f_{H^s(\R^d)}$.
Plancherel's identity and
$\partial_{\xi_j}\widehat f=-\ii\widehat{x_jf}$ therefore imply
\begin{equation}\label{eq:weighted-frequency-derivatives}
 \int_{\R^d}\ip\xi^{2s}
 \bigl(|\widehat f(\xi)|^2+|\nabla_\xi\widehat f(\xi)|^2\bigr)\dd\xi
 \le C\norm f_{H^s(\R^d)}^2.
\end{equation}
Define the $L^2(\Sph)$-valued function
\[
 \mathcal Q(r)(\omega)=r^{(d-1)/2}\ip r^s\widehat f(r\omega),
 \qquad r\ge1.
\]
Radial differentiation and polar coordinates give
$\mathcal Q\in H^1((1,\infty);L^2(\Sph))$ and
\[
 \int_1^\infty
 \bigl(\norm{\mathcal Q(r)}_{L^2(\Sph)}^2
       +\norm{\mathcal Q'(r)}_{L^2(\Sph)}^2\bigr)\dd r
 \le C\norm f_{H^s(\R^d)}^2.
\]

For $1\le k\le K_M^{\rm root}$, set
\[
 I_{k,M}=[\rho_{k,M}-\alpha_{k,M},\rho_{k,M}+\alpha_{k,M}],\qquad
 J_{k,M}=[\rho_{k,M}-1,\rho_{k,M}+1].
\]
For sufficiently large $M$, $I_{k,M}\subset J_{k,M}\subset(1,\infty)$,
and the intervals $J_{k,M}$ are disjoint because their centers are
separated by $\pi/\eps_M\to\infty$. Let
$\mathcal E_{k,M}=\norm{\mathcal Q}_{H^1(J_{k,M};L^2(\Sph))}^2$.
The one-dimensional Sobolev embedding \cite{Brezis2011} on intervals
of length two and the disjointness of the $J_{k,M}$ yield
\[
 \int_{I_{k,M}}\norm{\mathcal Q(r)}_{L^2(\Sph)}^2\dd r
 \le C\alpha_{k,M}\mathcal E_{k,M},\qquad
 \sum_{k=1}^{K_M^{\rm root}}\mathcal E_{k,M}\le C\norm f_{H^s(\R^d)}^2.
\]
For every real $0\le t\le s$, the prescribed half-width and
$\rho_{k,M}\le2M$ give
\[
 \alpha_{k,M}\rho_{k,M}^{2(t-s)}
 =c_sM^{-2s}\rho_{k,M}^{2t}
 \le c_s2^{2t}M^{2(t-s)}.
\]
Since $r\asymp\rho_{k,M}$ on $I_{k,M}$, polar coordinates and the
preceding estimates yield
\[
 \begin{aligned}
 \norm{f_M-f_M^\star}_{H^t(\R^d)}^2
 &\le C\sum_{k=1}^{K_M^{\rm root}}\rho_{k,M}^{2(t-s)}
       \int_{I_{k,M}}\norm{\mathcal Q(r)}_{L^2(\Sph)}^2\dd r\\
 &\le CM^{2(t-s)}\sum_{k=1}^{K_M^{\rm root}}\mathcal E_{k,M}
 \le CM^{2(t-s)}\norm f_{H^s(\R^d)}^2.
 \end{aligned}
\]
Here the integrals include the full last shell even if it extends
beyond the cutoff radius $M$.
Combining this estimate with \eqref{eq:hard-cutoff-error} and
\eqref{eq:extension}, then restricting to $\Omega$, proves the result.
\end{proof}

\subsection{Ridge representation}
Define the regularized inverse multiplier by
\begin{equation*}
 \mathfrak r_M(\tau)=\frac{\vartheta_M(\eps_M\tau)}{F_\sigma(\eps_M\tau)},
 \qquad |\tau|\le M,
\end{equation*}
where the quotient is defined to be zero at the zeros of the denominator.
The multiplier $\mathfrak r_M$ is even and satisfies
\begin{equation}\label{eq:multiplier-cancellation}
 F_\sigma(\eps_M\tau)\mathfrak r_M(\tau)
 =\vartheta_M(\eps_M\tau),\qquad |\tau|\le M.
\end{equation}

\begin{lemma}[Inverse multiplier bound]\label{lem:inverse-multiplier}
For all sufficiently large $M$,
\[
 \sup_{|\tau|\le M}|\mathfrak r_M(\tau)|
 \le
 \begin{cases}
  CM^{1/4},&\sigma=\sigma_G,\\[3pt]
  C\bigl(1+M^{1+\pi/2}/T_M\bigr),&\sigma=\sigma_D.
 \end{cases}
\]
\end{lemma}
\begin{proof}
For the Gaussian, \eqref{eq:activation-transforms} gives
\[
 \sup_{|\tau|\le M}|\mathfrak r_M(\tau)|
 =\pi^{-1/2}e^{\eps_M^2M^2/4}
 =\pi^{-1/2}(M+1)^{1/4}\le CM^{1/4}.
\]
For the tanh difference, evenness reduces the estimate to
$0\le\eta\le T_M$. The reciprocal $1/F_D$ is bounded on $[0,\pi/2]$.
For $\eta\in\mathcal I_k\setminus\mathcal Z_M$,
$1\le k\le K_M^{\rm root}$, the inequality
$|\sin\eta|\ge(2/\pi)|\eta-k\pi|\ge(2/\pi)\delta_{k,M}$ and
\eqref{eq:scaled-notch-widths} yield
\[
 |F_D(\eta)|^{-1}
 \le C\delta_{k,M}^{-1}e^{\pi^2k/2}
 \le CM T_M^{2s-1}k^{-2s}e^{\pi^2k/2}.
\]
To combine these bounds, we use the elementary estimate
\[
 \sum_{k=1}^K k^\alpha e^{bk}
 \le C_{\alpha,b}K^\alpha e^{bK},
 \qquad \alpha\in\R,\quad b>0,\quad K\ge1.
\]
Indeed, for $k>K/2$, use $k^\alpha\asymp K^\alpha$ and sum a
geometric series. For $k\le K/2$, the factor $e^{bK/2}$ absorbs
the remaining powers or logarithms of $K$.
With $\alpha=-2s$ and $b=\pi^2/2$, this estimate and the relations
$K_M^{\rm root}\asymp T_M$ and $\pi K_M^{\rm root}\le T_M+\pi/2$ give
\[
 \begin{aligned}
 \sup_{|\tau|\le M}|\mathfrak r_M(\tau)|
 &\le C\left(1+MT_M^{2s-1}(K_M^{\rm root})^{-2s}
                    e^{\pi^2K_M^{\rm root}/2}\right)\\
 &\le C\left(1+\frac{M}{T_M}e^{\pi T_M/2}\right).
 \end{aligned}
\]
Since $e^{\pi T_M/2}=(M+1)^{\pi/2}$, the tanh-difference bound follows.
\end{proof}

We now define the coefficient density by
\begin{equation}\label{eq:coefficient-density}
	 A_M[u](\omega,b)=\frac1{(2\pi)^d\eps_M}
	\int_{-M}^M\widehat f(\tau\omega)|\tau|^{d-1}
	\mathfrak r_M(\tau)e^{\ii\tau b}\dd\tau,\quad
	\omega\in\Sph,\quad b\in\R.
\end{equation}
The integral is absolutely convergent by \cref{lem:inverse-multiplier}
and the bound $\norm{\widehat f}_\infty\le\norm f_{L^1}$.
For fixed $u$, we write $A_M$ for $A_M[u]$
and set $g_{\omega,b}(x)=\sigma((\omega\cdot x-b)/\eps_M)$.

\begin{proposition}[Ridge representation]\label{prop:representation}
The modified function $f_M^\star$ admits the representation
\begin{equation}\label{eq:ridge-representation}
 f_M^\star(x)=\frac12\int_{\Sph}\int_\R A_M(\omega,b)g_{\omega,b}(x)\dd b\dd
 S(\omega),\qquad x\in\R^d.
\end{equation}
Moreover,
\begin{equation}\label{eq:joint-symmetry}
 A_M(-\omega,-b)=A_M(\omega,b),\qquad g_{-\omega,-b}=g_{\omega,b},
\end{equation}
and $A_M$ is real for real $u$.
On bounded spatial sets, derivatives of every order with respect to
$x$ may be taken under the integral sign.
\end{proposition}
\begin{proof}
The identity
\begin{equation*}
 \int_\R\sigma\left(\frac{t-b}{\eps_M}\right)e^{\ii\tau b}\dd b
 =\eps_MF_\sigma(\eps_M\tau)e^{\ii\tau t}
\end{equation*}
and \eqref{eq:multiplier-cancellation} reduce the right-hand side of
\eqref{eq:ridge-representation} to
\[
 \frac1{2(2\pi)^d}\int_{\Sph}\int_{-M}^M
 \widehat f(\tau\omega)\vartheta_M(\eps_M\tau)
       e^{\ii\tau\omega\cdot x}|\tau|^{d-1}\dd\tau\dd S(\omega),
\]
which is $f_M^\star(x)$ by polar coordinates with a signed radial variable.
Fubini's theorem applies because the integral of the absolute value is bounded by
\[
 C\norm{\mathfrak r_M}_{L^\infty([-M,M])}\norm\sigma_{L^1}
 \int_{|\xi|\le M}|\widehat f(\xi)|\dd\xi<\infty.
\]
For a multi-index $\alpha$,
\begin{equation*}
 \partial_x^\alpha g_{\omega,b}(x)
 =\omega^\alpha\eps_M^{-|\alpha|}
   \sigma^{(|\alpha|)}\left(\frac{\omega\cdot x-b}{\eps_M}\right).
\end{equation*}
After substitution of \eqref{eq:coefficient-density}, the frequency
factors are bounded on $[-M,M]$, while the rapid decay of
$\sigma^{(|\alpha|)}$ yields an integrable majorant in $b$, uniform
for $\omega\in\Sph$ and $x$ in bounded sets. Dominated convergence
therefore justifies differentiation to any finite order.
Finally, evenness of $\mathfrak r_M$ and $\sigma$ gives
\eqref{eq:joint-symmetry}, and conjugate symmetry of $\widehat f$ shows
that $A_M$ is real when $u$ is real.
\end{proof}

\section{Discretization of the ridge representation}\label{sec:discretization}
The ridge representation in \cref{prop:representation} involves
continuous direction and offset variables. We first discretize the
direction integral by spherical cubature. We next derive coefficient
and discrete Fourier estimates to control the midpoint error in the
offset variable through Poisson summation. We truncate the resulting
grid at the cutoff $H_\Omega(\omega_i)+\Delta_{\sigma,M}$, with a correction
at the last retained node for the tanh difference. The angular, midpoint,
and truncation errors are combined with the frequency approximation
error in the final proof.

\subsection{Angular discretization}
\begin{lemma}[Spherical cubature]\label{lem:sphere-rule}
For each node $\omega_i\in\Theta_M^{(d)}$, let $w_i$ be the product
weight obtained by assigning
weight $\pi/L_M$ to each initial circle node and multiplying by the
corresponding Gauss--Jacobi weight at each recursive step.
The node set is antipodally symmetric, and the weights are positive
and equal at paired nodes.
For every integer $0\le q_M\le2L_M-1$, the resulting cubature rule is exact on
$\mathcal P_{q_M}(\Sph)$, the space of restrictions to $\Sph$ of
polynomials of total degree at most $q_M$. In particular,
\begin{equation}\label{eq:sphere-exactness}
 \sum_iw_i p(\omega_i)=\int_{\Sph}p(\omega)\dd S(\omega)
 \quad(p\in\mathcal P_{q_M}(\Sph)),\qquad \sum_iw_i=|\Sph|.
\end{equation}
\end{lemma}
\begin{proof}
The circle nodes have equal positive weights. Each Gauss--Jacobi rule
has positive weights, and reflection $t\mapsto-t$ preserves its nodes
and weights because the Jacobi weight is even. Taking products at each
step of the recursion therefore gives equal weights to antipodal nodes.

The $2L_M$-point circle rule is exact for trigonometric polynomials
of degree at most $2L_M-1$. In the recursive step, let $a$ be the total
degree of a monomial in the first $k-1$ coordinates and $b$ its degree
in the last coordinate $t$, with $a+b\le q_M$. At fixed $t$, if $a$
is odd, the monomial changes sign under $\zeta\mapsto-\zeta$ on
$\mathbb S^{k-2}$. Hence the inner integral vanishes, as does the
inner cubature sum, since antipodal nodes have equal weights.
If $a$ is even, the inner rule is exact by induction, and the remaining
factor $(1-t^2)^{a/2}t^b$ is a polynomial of degree $a+b$.
Gauss--Jacobi exactness completes the induction. Taking $p=1$ gives
the total weight in \eqref{eq:sphere-exactness}. See also \cite{DaiXu2013}.
\end{proof}

For a product grid with parameter $L_M$, let $q_M$ be an integer with
$0\le q_M\le 2L_M-1$. Replacing the sphere integral in
\eqref{eq:ridge-representation} by the cubature sum gives the angular
approximation
\begin{equation}\label{eq:angular-operator}
 \mathcal S_{M,q_M}u(x)
 =\frac12\sum_{i=1}^{J_M}w_i
 \int_\R A_M[u](\omega_i,b)
          \sigma\!\left(\frac{\omega_i\cdot x-b}{\eps_M}\right)\dd b.
\end{equation}

\begin{lemma}[Angular cubature error]\label{lem:angular}
Let $q_M$ be an integer with $m_*\le q_M\le 2L_M-1$, and set
$B=(R_\Omega+R_0)\sinh(1/2)$.
For all sufficiently large $M$, the angular cubature error satisfies
\begin{equation*}
 \norm{f_M^\star-\mathcal S_{M,q_M}u}_{H^m(\Omega)}
 \le CM^{d+m}e^{BM-q_M/2}\norm u_{H^s(\Omega)},
\end{equation*}
for every $u\in H^s(\Omega)$ and every integer $0\le m\le m_*$.
The constant $C$ is independent of $M$, $q_M$, $L_M$, and $u$.
\end{lemma}
\begin{proof}
In \eqref{eq:angular-operator}, the direction
$\omega_i$ enters both the coefficient density and the ridge function.
Substituting \eqref{eq:coefficient-density} gives an offset integral
for each frequency $\tau$. The change of variables
$v=(\omega_i\cdot x-b)/\eps_M$ yields
\[
 \int_\R \sigma\!\left(\frac{\omega_i\cdot x-b}{\eps_M}\right)
          e^{\ii\tau b}\dd b
 =\eps_M F_\sigma(\eps_M\tau)e^{\ii\tau\omega_i\cdot x}.
\]
Using \eqref{eq:multiplier-cancellation} and writing
$\widehat f(\tau\omega_i)=\int_{\R^d}f(y)e^{-\ii\tau\omega_i\cdot y}\dd y$,
we obtain
\begin{equation}\label{eq:angular-plane-wave}
 \begin{aligned}
 \mathcal S_{M,q_M}u(x)
 &=\frac{1}{2(2\pi)^d}\sum_iw_i
   \int_{-M}^{M}|\tau|^{d-1}\vartheta_M(\eps_M\tau)
   \widehat f(\tau\omega_i)e^{\ii\tau\omega_i\cdot x}\dd\tau\\
 &=\frac{1}{2(2\pi)^d}\sum_iw_i
   \int_{-M}^{M}|\tau|^{d-1}\vartheta_M(\eps_M\tau)
   \int_{\R^d}f(y)e^{\ii\tau\omega_i\cdot(x-y)}\dd y\dd\tau.
 \end{aligned}
\end{equation}

For $f_M^\star$, write the inverse Fourier integral in signed polar
coordinates $\xi=\tau\omega$. Since $f\in L^1(\R^d)$ and $|\tau|\le M$,
differentiation under the integral gives, for $|\alpha|\le m$,
\[
 \begin{aligned}
 \partial^\alpha f_M^\star(x)
 &=\frac{1}{2(2\pi)^d}
   \int_{-M}^M(\ii\tau)^{|\alpha|}|\tau|^{d-1}
       \vartheta_M(\eps_M\tau)\\
 &\quad\times\int_{\R^d}f(y)
   \left[\int_{\Sph}\omega^\alpha
      e^{\ii\tau\omega\cdot(x-y)}\dd S(\omega)\right]\dd y\dd\tau.
 \end{aligned}
\]
By \eqref{eq:angular-plane-wave}, applying cubature to the bracketed
integral gives $\partial^\alpha\mathcal S_{M,q_M}u(x)$.
Since $\omega^\alpha$ has degree $|\alpha|\le m_*\le q_M$,
we approximate the plane wave to degree $q_M-|\alpha|$.
At $x=y$ the plane wave is constant and cubature is exact.
For $x\in\Omega$ and $y\in\supp f$ with $x\ne y$, set
$z=\omega\cdot(x-y)/|x-y|$. On the Bernstein ellipse with parameter
$e^{1/2}$, we have $|\operatorname{Im}z|\le\sinh(1/2)$, so
$e^{\ii\tau|x-y|z}$ is bounded in modulus by
$e^{|\tau||x-y|\sinh(1/2)}$. For the Chebyshev truncation $P$ of the
prescribed degree, the uniform error estimate \cite{Trefethen2013} gives
\[
 \norm{e^{\ii\tau|x-y|z}-P(z)}_{L^\infty([-1,1])}
 \le C\exp\left\{|\tau||x-y|\sinh(1/2)
                     -\frac{q_M-|\alpha|}{2}\right\}.
\]
Since $|\tau|\le M$ and $|x-y|\le R_\Omega+R_0$, the exponent is at most
$BM-q_M/2+|\alpha|/2$. The polynomial
$\omega^\alpha P(\omega\cdot(x-y)/|x-y|)$ belongs to $\mathcal P_{q_M}(\Sph)$.
The integral and the cubature sum agree on this polynomial, so only
the approximation error contributes to their difference. Positivity
of the weights and \eqref{eq:sphere-exactness} give
\begin{equation}\label{eq:angular-cubature-bound}
 \begin{aligned}
 &\left|\int_{\Sph}\omega^\alpha
              e^{\ii\tau\omega\cdot(x-y)}\dd S(\omega)
       -\sum_iw_i\omega_i^\alpha
              e^{\ii\tau\omega_i\cdot(x-y)}\right|\\
 &\quad\le2|\Sph|\norm{e^{\ii\tau|x-y|z}-P(z)}_{L^\infty([-1,1])}
 \le Ce^{BM-q_M/2}.
 \end{aligned}
\end{equation}
Here $e^{|\alpha|/2}\le e^{m_*/2}$ is absorbed into $C$.
Using \eqref{eq:angular-cubature-bound} in the $y$ and $\tau$ integrals
and noting that $|\vartheta_M|\le1$, we obtain
\[
 \begin{aligned}
 \norm{\partial^\alpha(f_M^\star-\mathcal S_{M,q_M}u)}_{L^\infty(\Omega)}
 &\le Ce^{BM-q_M/2}\norm f_{L^1}
                 \int_{-M}^M|\tau|^{d-1+|\alpha|}\dd\tau\\
 &\le CM^{d+|\alpha|}e^{BM-q_M/2}\norm u_{H^s(\Omega)}.
 \end{aligned}
\]
Since $\Omega$ is bounded, taking $L^2(\Omega)$ norms and summing over
$|\alpha|\le m$ proves the estimate.
\end{proof}

To obtain the target order $M^{m-s}$ from \cref{lem:angular}, we need
to bound $M^{d+s}e^{BM-q_M/2}$ uniformly in $M$. It suffices to take
$q_M\ge 2BM+2(d+s)\log M$. Since $2(d+s)\log M\le M$ for all
sufficiently large $M$, we choose
\[
 C_{\rm ang}=2B+1=2(R_\Omega+R_0)\sinh(1/2)+1,\qquad
 q_M=\lceil C_{\rm ang}M\rceil.
\]
This choice gives $BM-q_M/2\le-M/2$, so the exponential decay absorbs
the factor $M^{d+s}$. We take $L_M=\lceil(q_M+1)/2\rceil$, the smallest
integer satisfying $q_M\le 2L_M-1$. For large $M$
we also have $q_M\ge m_*$. Thus \cref{lem:angular} gives, for every
integer $0\le m\le m_*$,
\begin{equation}\label{eq:angular-target}
 \norm{f_M^\star-\mathcal S_{M,q_M}u}_{H^m(\Omega)}
 \le CM^{m-s}\norm u_{H^s(\Omega)}.
\end{equation}
The circle rule has $2L_M$ nodes, and each of the $d-2$ recursive steps
multiplies the node count by $L_M$. Since $L_M\asymp q_M\asymp M$, the
number of directions is
\begin{equation}\label{eq:sphere-count}
 J_M=2L_M^{d-1}\asymp M^{d-1}.
\end{equation}

\Needspace{8\baselineskip}
\subsection{Offset discretization}\label{subsec:offset-discretization}
After angular discretization, we approximate the remaining offset
integrals in $\mathcal S_{M,q_M}u$ by the midpoint rule on $\R$.

\subsubsection{Coefficient and radial energy estimates}
For fixed sufficiently large $M$, $\mathfrak r_M$ is bounded on $[-M,M]$.
Differentiating \eqref{eq:coefficient-density} $k$ times in $b$
multiplies the integrand by $(\ii\tau)^k$, so the bound
$\norm{\widehat f}_\infty\le\norm f_{L^1}\le C\norm u_{H^s(\Omega)}$
gives, for every integer $k\ge0$,
\begin{equation}\label{eq:fixed-resolution-density-bound}
 \sup_{\omega\in\Sph,\,b\in\R}
       |\partial_b^kA_M[u](\omega,b)|
 \le C_{M,k}\norm u_{H^s(\Omega)},
\end{equation}
where $C_{M,k}$ is independent of $u$.

For $k=0$, we can track the dependence on $M$ for each direction
$\omega\in\Sph$. Weighted Cauchy--Schwarz in \eqref{eq:coefficient-density}
and the evenness of $\mathfrak r_M$ give
\begin{equation}\label{eq:coefficient-envelope}
 \begin{aligned}
 \sup_{b\in\R}|A_M(\omega,b)|
 &\le \frac{\sqrt{2}}{(2\pi)^d}
 \underbrace{\left(\int_{-M}^M |\tau|^{d-1}\ip{\tau}^{2s}
       |\widehat f(\tau\omega)|^2\dd\tau\right)^{1/2}}_{E_M(\omega)}\\
 &\quad\times
 \underbrace{\frac{1}{\eps_M}\left(\int_0^M \tau^{d-1}\ip{\tau}^{-2s}
       |\mathfrak r_M(\tau)|^2\dd\tau\right)^{1/2}}_{\Gamma_M}.
 \end{aligned}
\end{equation}
Here $E_M(\omega)$ is the radial energy of $f$ in direction $\omega$.
Only $\Gamma_M$ depends on the activation, through the scale $\eps_M$
and the regularized inverse multiplier $\mathfrak r_M$.

\begin{lemma}\label{lem:budgets}
For all sufficiently large $M$, the factor $\Gamma_M$ in
\eqref{eq:coefficient-envelope} satisfies
\begin{equation}\label{eq:Gamma-estimates}
 \Gamma_M\le
 \begin{cases}
 C\left(\dfrac M{\sqrt{T_M}}+
            \dfrac{M^{d/2+5/4-s}}{T_M}\right),&\sigma=\sigma_G,\\[6pt]
 C\left(\dfrac M{T_M}+
            \dfrac{M^{d/2+3/2+\pi/2-s}}{T_M^2}\right),&\sigma=\sigma_D
 \end{cases},
\end{equation}
where $C$ is independent of $M$.
\end{lemma}
The proof is given in Appendix~\ref{app:coefficient-bounds}.

For the radial energies $E_{M,i}=E_M(\omega_i)$ in
\eqref{eq:coefficient-envelope}, we use the $\ell^2$ norm with
cubature weights $w_i$.

\begin{lemma}\label{lem:sampled-energy}
The radial energies satisfy
\begin{equation}\label{eq:sampled-energy}
 \sum_iw_iE_{M,i}^2
 \le C\left[1+\left(\frac M{q_M}\right)^{d-1}\right]
       \norm f_{H^s(\R^d)}^2
 \le C\norm u_{H^s(\Omega)}^2.
\end{equation}
\end{lemma}

The proof is given in Appendix~\ref{app:fourier-estimates}.

\subsubsection{Midpoint discretization}
Fix a sufficiently large $M$ and a direction $\omega\in\Sph$.
By \eqref{eq:fixed-resolution-density-bound}, $A_M(\omega,\cdot)$ has
bounded derivatives of every order. Since $\sigma$ and its derivatives
decay rapidly, the product rule shows that the offset integrand
$b\mapsto A_M(\omega,b)\sigma((t-b)/\eps_M)$ is a Schwartz function
for each $t\in\R$. Its samples on a uniform grid are therefore
absolutely summable.
For $h_M>0$ and $j\in\mathbb Z$, let $b_{j,M}=(j-\tfrac12)h_M$
be the midpoint of $[(j-1)h_M,jh_M]$.
Replacing each offset integral in  \eqref{eq:angular-operator} by its
midpoint sum gives
\begin{equation}\label{eq:infinite-operator}
 \mathcal A^\infty_{M,q_M,h_M}u(x)
 =\frac{h_M}2\sum_{i=1}^{J_M}w_i\sum_{j\in\mathbb Z}
 A_M[u](\omega_i,b_{j,M})
 \sigma\!\left(\frac{\omega_i\cdot x-b_{j,M}}{\eps_M}\right).
\end{equation}
By \eqref{eq:fixed-resolution-density-bound} and the decay of
the derivatives of $\sigma$, this series and its spatial derivatives
converge absolutely and locally uniformly in $x$.

Let $\lambda_{h_M}=2\pi/h_M$ denote the frequency spacing of the
midpoint lattice.
At a fixed radial frequency $\tau$ in \eqref{eq:coefficient-density},
the midpoint rule samples $G_{\tau,t}(b)=\sigma((t-b)/\eps_M)e^{\ii\tau b}$
with $t=\omega_i\cdot x$. Its Fourier transform in $b$ is
\[
 \widehat G_{\tau,t}(\xi)=\eps_M F_\sigma(\eps_M(\tau-\xi))
                          e^{\ii(\tau-\xi)t}.
\]
Since $G_{\tau,t}$ is a Schwartz function, Poisson summation on the
midpoint lattice \cite{Katznelson2004} gives
$h_M\sum_jG_{\tau,t}(b_{j,M})=
\sum_{\ell\in\mathbb Z}(-1)^\ell\widehat G_{\tau,t}(\ell\lambda_{h_M})$.
Replacing $\ell$ by $-\ell$ yields
\begin{equation}\label{eq:poisson-id}
 h_M\sum_{j\in\mathbb Z}\sigma\!\left(\frac{t-b_{j,M}}{\eps_M}\right) e^{\ii\tau b_{j,M}}
 =\eps_M\sum_{\ell\in\mathbb Z}(-1)^\ell F_\sigma(\eps_M(\tau+\ell\lambda_{h_M}))
 e^{\ii(\tau+\ell\lambda_{h_M})t}.
\end{equation}

In \eqref{eq:coefficient-density}, the Fourier mode $e^{\ii\tau b}$ carries
the factor $\mathfrak r_M(\tau)/\eps_M$.
For $\ell=0$ in \eqref{eq:poisson-id}, this factor cancels
$\eps_M$ and leaves
$\mathfrak r_M(\tau)F_\sigma(\eps_M\tau)=\vartheta_M(\eps_M\tau)$
by \eqref{eq:multiplier-cancellation}. The unshifted term therefore
recovers $\mathcal S_{M,q_M}u$ after integration in $\tau$ and summation
over the directions, as in \eqref{eq:angular-plane-wave}.
For $\ell\ne0$, the multiplier for the shifted frequency is
\[
 Q_\ell(\tau)=\vartheta_M(\eps_M\tau)
       \frac{F_\sigma(\eps_M(\tau+\ell\lambda_{h_M}))}
            {F_\sigma(\eps_M\tau)},
\]
on $|\tau|\le M$, with $Q_\ell=0$ on the excluded frequencies.
Denote the contribution of this shift, apart from the midpoint factor
$(-1)^\ell$, by
\[
 \mathcal E_\ell u(x)
 =\frac1{2(2\pi)^d}\sum_iw_i\int_{-M}^{M}
 |\tau|^{d-1}\widehat f(\tau\omega_i)Q_\ell(\tau)
 e^{\ii(\tau+\ell\lambda_{h_M})\omega_i\cdot x}\dd\tau.
\]

For a positive shift $\ell$, the radial frequencies in $\mathcal E_\ell u$
lie in $[\lambda-M,\lambda+M]$ with $\lambda=\ell\lambda_{h_M}$.
When the lower endpoint is at least $M$, the following estimate bounds
the $H^m(\Omega)$ norm of the sum over directions in terms of the
amplitudes on the individual rays.

\begin{lemma}\label{lem:shifted-synthesis}
For sufficiently large $M$, let $\lambda\ge2M$ and let $z=(z_i)_{i=1}^{J_M}$ satisfy
$\sum_iw_i\int_{-M}^M|z_i(\tau)|^2\dd\tau<\infty$.
Define
\[
 \mathcal T_\lambda z(x)=\sum_iw_i\int_{-M}^{M}z_i(\tau)
                         e^{\ii(\lambda+\tau)\omega_i\cdot x}\dd\tau.
\]
Then, for every integer $0\le m\le m_*$,
\begin{equation}\label{eq:shifted-synthesis}
 \begin{aligned}
 \norm{\mathcal T_\lambda z}_{H^m(\Omega)}
 &\le C(\lambda+M)^m
       \big[(\lambda-M)^{-(d-1)/2}+q_M^{-(d-1)/2}\big]\\
 &\quad\times
 \left(\sum_iw_i\int_{-M}^{M}|z_i(\tau)|^2\dd\tau\right)^{1/2},
 \end{aligned}
\end{equation}
where $C$ is independent of $\lambda$ and $z$.
\end{lemma}
The proof is given in Appendix~\ref{app:fourier-estimates}.

\begin{lemma}\label{lem:poisson-alias}
For each sufficiently large $M$, every $h_M>0$, and every integer
$0\le m\le m_*$, the infinite-grid error has the representation
\begin{equation}\label{eq:alias-expansion}
 \mathcal A^\infty_{M,q_M,h_M}u-\mathcal S_{M,q_M}u
 =\sum_{\ell\ne0}(-1)^\ell\mathcal E_\ell u.
\end{equation}
The series converges in $H^m(\Omega)$, and
$\mathcal E_{-\ell}u=\mathcal E_\ell u$ for every $\ell\ne0$.
If the relative frequency spacing satisfies $\kappa_M:=\lambda_{h_M}/M\ge2$,
then for every integer $\ell\ge1$,
\begin{equation}\label{eq:single-alias-bound}
 \begin{aligned}
 \norm{\mathcal E_\ell u}_{H^m(\Omega)}
 &\le C(\ell\,\kappa_M+1)^mM^{m-(d-1)/2}\\
 &\quad\times\sup_{|\tau|\le M}
 \{|\tau|^{(d-1)/2}\ip{\tau}^{-s}|Q_\ell(\tau)|\}
 \norm u_{H^s(\Omega)}.
 \end{aligned}
\end{equation}
The constant $C$ is independent of $\ell$.
\end{lemma}
\begin{proof}
For fixed $M$, boundedness of $\mathfrak r_M$ on the retained band
and $\norm{\widehat f}_\infty\le\norm f_{L^1}$ ensure absolute
integrability of the coefficient integrand in
\eqref{eq:coefficient-density}, uniformly in $i$.
For every integer $k\ge0$ and every fixed $R'>0$,
Schwartz decay of $\sigma$ and $F_\sigma$ gives
\[
 \sup_{|t|\le R'}\eps_M^{-k}
 \left|\sigma^{(k)}\!\left(\frac{t-b_{j,M}}{\eps_M}\right)\right|
 \le C_{M,h_M,R',k}(1+|j|)^{-2}
\]
and
\[
 \sup_{|\tau|\le M}\eps_M(1+|\tau+\ell\lambda_{h_M}|)^k
       |F_\sigma(\eps_M(\tau+\ell\lambda_{h_M}))|
 \le C_{M,h_M,k}(1+|\ell|)^{-2}.
\]
Set $t=\omega_i\cdot x$ and multiply \eqref{eq:poisson-id} by
$|\tau|^{d-1}\widehat f(\tau\omega_i)\mathfrak r_M(\tau)/[(2\pi)^d\eps_M]$.
The bounds above justify termwise integration in $\tau$ and, since
$|\omega_i^\alpha|\le1$, termwise differentiation in $x$.
After summation with weights $w_i/2$, the series and their spatial
derivatives converge absolutely and uniformly on bounded sets.
Since $\Omega$ is bounded, the expansion
converges in $H^m(\Omega)$. Subtracting the unshifted term gives
\eqref{eq:alias-expansion}.
Equal weights at antipodal nodes, evenness of $F_\sigma$ and $\vartheta_M$, and
the substitution $(\omega_i,\tau)\mapsto(-\omega_i,-\tau)$ give
$\mathcal E_{-\ell}u=\mathcal E_\ell u$.
Consequently,
\[
 \norm{\mathcal A^\infty_{M,q_M,h_M}u-\mathcal S_{M,q_M}u}_{H^m(\Omega)}
 \le2\sum_{\ell\ge1}\norm{\mathcal E_\ell u}_{H^m(\Omega)}.
\]

To estimate a single term, suppose now that $\kappa_M\ge2$ and fix
$\ell\ge1$. The radial frequencies are then at least
$\ell\lambda_{h_M}-M\ge M$. With
$z_i(\tau)=|\tau|^{d-1}\widehat f(\tau\omega_i)Q_\ell(\tau)$, we have
\[
 \mathcal E_\ell u=\frac{1}{2(2\pi)^d}\mathcal T_{\ell\lambda_{h_M}}z.
\]
By \cref{lem:shifted-synthesis}, it suffices to estimate the weighted
$L^2$ norm of these amplitudes. The factorization
\[
 |z_i(\tau)|^2
 =\bigl(|\tau|^{d-1}\ip\tau^{2s}|\widehat f(\tau\omega_i)|^2\bigr)
  \bigl(|\tau|^{d-1}\ip\tau^{-2s}|Q_\ell(\tau)|^2\bigr)
\]
shows that
\[
 \sum_iw_i\int_{-M}^{M}|z_i(\tau)|^2\dd\tau
 \le \sup_{|\tau|\le M}
       \{|\tau|^{d-1}\ip{\tau}^{-2s}|Q_\ell(\tau)|^2\}
       \sum_iw_iE_{M,i}^2.
\]
The sampled energy on the right is bounded by \eqref{eq:sampled-energy}.
The frequency factor in \eqref{eq:shifted-synthesis} satisfies
\[
 \begin{aligned}
 &(\ell\lambda_{h_M}+M)^m
       \big[(\ell\lambda_{h_M}-M)^{-(d-1)/2}+q_M^{-(d-1)/2}\big]\\
 &\quad=(\ell\,\kappa_M+1)^mM^{m-(d-1)/2}
       \big[(\ell\,\kappa_M-1)^{-(d-1)/2}+(M/q_M)^{(d-1)/2}\big]\\
 &\quad\le C(\ell\,\kappa_M+1)^mM^{m-(d-1)/2},
 \end{aligned}
\]
because $\ell\,\kappa_M\ge2$ and $q_M\ge C_{\rm ang}M$.
Substituting the amplitude and frequency bounds in \eqref{eq:shifted-synthesis}
proves \eqref{eq:single-alias-bound}.
\end{proof}

\Needspace{6\baselineskip}
We choose the spacing by estimating the weighted multiplier in
\eqref{eq:single-alias-bound} and summing over $\ell\ge1$.

\begin{lemma}\label{lem:poisson}
For each activation there is an explicit choice of $h_M\asymp M^{-1}$,
independent of $u$ and $m$, for which
\begin{equation}\label{eq:alias-target}
 \norm{\mathcal A^\infty_{M,q_M,h_M}u-\mathcal S_{M,q_M}u}_{H^m(\Omega)}
 \le C M^{m-s}\norm u_{H^s(\Omega)}
\end{equation}
for every $u\in H^s(\Omega)$, every integer $0\le m\le m_*$, and all
sufficiently large $M$.
\end{lemma}
\begin{proof}
For $\kappa_M=2\pi/(Mh_M)\ge2$, we estimate the sum of the nonzero
shifts in terms of $\kappa_M$.

\emph{Gaussian activation.}
Since $\eps_M^2=T_M/M^2$ and $\vartheta_M=1$,
\[
 Q_\ell(\tau)
 =\exp\!\left[-\frac{T_M}4\left((\ell\,\kappa_M)^2
                                  +2\ell\,\kappa_M\frac\tau M\right)\right].
\]
For $\ell\ge1$ and fixed $r=|\tau|$, $Q_\ell(\tau)$ is largest at $\tau=-r$.
On $0\le r\le1$, the weighted expression in
\eqref{eq:single-alias-bound} is bounded by
$e^{-(\ell\,\kappa_M)^2T_M/4+\ell\,\kappa_M T_M/(2M)}$.
On $1\le r\le M$, use $r^{(d-1)/2}\ip r^{-s}\le r^{(d-1)/2-s}$.
For any real $\gamma$ and $c>0$, the logarithmic derivative of
$r^\gamma e^{cr}$ shows that its maximum on a compact positive interval
is attained at an endpoint. Thus
\[
 \begin{aligned}
 &\sup_{|\tau|\le M}|\tau|^{(d-1)/2}\ip{\tau}^{-s}|Q_\ell(\tau)|\\
 &\quad\le C\big[e^{-(\ell\,\kappa_M)^2T_M/4+\ell\,\kappa_M T_M/(2M)}
 +M^{(d-1)/2-s}e^{-\ell\,\kappa_M(\ell\,\kappa_M-2)T_M/4}\big].
 \end{aligned}
\]
To sum over $\ell$, compare the decay exponents with their values
at $\ell=1$. For $\ell\ge2$, the differences satisfy
\[
 \begin{aligned}
 (\ell-1)\left[\frac{\kappa_M^2(\ell+1)}4
                        -\frac{\kappa_M}{2M}\right]T_M
 &\ge\ell(\ell-1)T_M,\\
 \frac{\kappa_M(\ell-1)[\kappa_M(\ell+1)-2]}4T_M
 &\ge\ell(\ell-1)T_M.
 \end{aligned}
\]
These inequalities hold for $\kappa_M\ge2$ and all sufficiently large
$M$. Dividing each weighted sum by its exponential at $\ell=1$ and
using $1+\ell\kappa_M\le\ell(1+\kappa_M)$ gives the upper bound
\[
 (1+\kappa_M)^m\left(1+\sum_{\ell\ge2}\ell^m
             e^{-\ell(\ell-1)T_M}\right)\le C(1+\kappa_M)^m.
\]
At $\ell=1$, the extra factor in the first exponential satisfies
$e^{\kappa_M T_M/(2M)}\le e^{\kappa_M/2}$. Set
$C_\kappa=C(1+\kappa)^{m_*}e^{\kappa/2}$. Summing
\eqref{eq:single-alias-bound} and using $e^{-\gamma T_M}\le M^{-\gamma}$
for $\gamma\ge0$, we obtain
\begin{equation}\label{eq:alias-G}
 \begin{aligned}
 &\norm{\mathcal A^\infty_{M,q_M,h_M}u-\mathcal S_{M,q_M}u}_{H^m(\Omega)}\\
 &\quad\le C_{\kappa_M}\left[M^{m-(d-1)/2-\kappa_M^2/4}
          +M^{m-s-\kappa_M(\kappa_M-2)/4}\right]\norm u_{H^s(\Omega)}.
 \end{aligned}
\end{equation}
For $\kappa_M\ge2$, the second term in brackets is at most $M^{m-s}$.
The first satisfies the same bound if
$\kappa_M^2/4\ge s-(d-1)/2$. We therefore take
\[
 h_{G,M}=\frac{\pi}{M\sqrt{\max\{1,s-(d-1)/2\}}}.
\]
With this choice $\kappa_M$ depends only on $d$ and $s$, and
\eqref{eq:alias-G} gives \eqref{eq:alias-target}.

\emph{Tanh-difference activation.}
Fix $\kappa_M=\kappa\ge2$ and set $\eta=\eps_M\tau$.
Since $\eps_M=T_M/M$, the shifted argument of $F_D$ is
$\eta+\ell\kappa T_M\ge T_M$ for $\ell\ge1$ and $|\eta|\le T_M$.
Bounding the shifted sine by one in \eqref{eq:activation-transforms} gives
\[
 |F_D(\eta+\ell\kappa T_M)|
 \le C e^{-\pi\ell\kappa T_M/2+\pi|\eta|/2}.
\]
The reciprocal $1/F_D$ is bounded on $[-\pi/2,\pi/2]$.
Together with $|\tau|^{(d-1)/2}\ip{\tau}^{-s}\le 1+M^{(d-1)/2-s}$
for $|\tau|\le M$, this gives
\[
 |\tau|^{(d-1)/2}\ip{\tau}^{-s}|Q_\ell(\tau)|
 \le C(1+M^{(d-1)/2-s})e^{-\pi\ell\kappa T_M/2}
 \qquad (|\eta|\le\pi/2).
\]

On the excluded intervals $Q_\ell(\tau)=0$. For a retained frequency
with $r=|\eta|\in[\pi/2,T_M]$, choose $k$ such that $r\in\mathcal I_k$. Then $\delta_{k,M}\le |r-k\pi|\le\pi/2$, so
$|\sin\eta|\ge (2/\pi)|r-k\pi|$. Since $k\pi\asymp r$, the half-width
in \eqref{eq:scaled-notch-widths} gives
\[
 \frac{1}{|\sin\eta|}\le\frac{\pi}{2\delta_{k,M}}
 \le C\frac M{T_M}\left(\frac{T_M}{r}\right)^{2s}.
\]
Using $|F_D(\eta)|^{-1}\le C e^{\pi r/2}/|\sin\eta|$ and
$|\tau|=Mr/T_M$ in the weighted multiplier, we obtain
\[
 |\tau|^{(d-1)/2}\ip{\tau}^{-s}|Q_\ell(\tau)|
 \le C e^{-\pi\ell\kappa T_M/2}\,
       \frac{M^{1+(d-1)/2-s}}{T_M}
       \left(\frac{T_M}{r}\right)^{3s-(d-1)/2}e^{\pi r}.
\]
The bound $(T_M/r)^{3s-(d-1)/2}e^{\pi r}\le C e^{\pi T_M}$ follows since
$1\le T_M/r\le2$ for $r\ge T_M/2$, while for $r\le T_M/2$ the
factor $e^{\pi(T_M-r)}$ dominates any fixed power of $T_M$.

We insert the two frequency bounds into \eqref{eq:single-alias-bound}
and sum over the shifts in \eqref{eq:alias-expansion}. Since
\[
 \sum_{\ell\ge1}(1+\ell\kappa)^m e^{-\pi\ell\kappa T_M/2}
 \le C(1+\kappa)^m e^{-\pi\kappa T_M/2},
\]
the identity $T_M=\log(M+1)$ gives
\begin{equation}\label{eq:alias-D}
 \begin{aligned}
 &\norm{\mathcal A^\infty_{M,q_M,h_M}u-\mathcal S_{M,q_M}u}_{H^m(\Omega)}\\
 &\quad\le C(1+\kappa)^m M^{m-s}
 \left(M^{s-(d-1)/2-\pi\kappa/2}+\frac{M^{1+\pi-\pi\kappa/2}}{T_M}\right)
 \norm u_{H^s(\Omega)}.
 \end{aligned}
\end{equation}
Both exponents in parentheses are nonpositive when
$\pi\kappa/2\ge\max\{\pi+1,\,s-(d-1)/2\}$. Choosing $\kappa$ at
equality gives $\kappa>2$ and, by $h_M=2\pi/(\kappa M)$,
\begin{equation}\label{eq:offset-spacing-choice}
 h_{D,M}=\frac{\pi^2}{M\max\{\pi+1,\,s-(d-1)/2\}}.
\end{equation}
With this choice, \eqref{eq:alias-D} proves \eqref{eq:alias-target}.
\end{proof}

\subsection{Support truncation and boundary correction}
\label{subsec:support-truncation}

We now truncate the infinite-grid approximation to obtain an element
of $V_M^\Omega$. For $j\ge1$, the positive node $b_{j,M}$ is paired
with $b_{1-j,M}=-b_{j,M}$. By \eqref{eq:joint-symmetry} and the equal
weights at antipodal cubature nodes, the positive and negative offsets
in \eqref{eq:infinite-operator} give equal contributions. Hence
\begin{equation}\label{eq:positive-grid-folding}
 \mathcal A^\infty_{M,q_M,h_M}u
 =h_M\sum_iw_i\sum_{j\ge1}A_M(\omega_i,b_{j,M})g_{\omega_i,b_{j,M}}.
\end{equation}

We take the margins $\Delta_M$ in \eqref{eq:parameters} with suitable
positive constants $a_G$ and $a_D$. In direction $\omega_i$ we retain
the positive offsets $b_{j,M}\le H_\Omega(\omega_i)+\Delta_M$.
For all sufficiently large $M$, $K_{i,M}\ge1$ for every direction
$\omega_i$. The last retained node
\mbox{$b_i^*=(K_{i,M}-1/2)h_M$} satisfies
\begin{equation}\label{eq:last-node}
 H_\Omega(\omega_i)+\Delta_M-h_M
 <b_i^*\le H_\Omega(\omega_i)+\Delta_M.
\end{equation}

For the tanh difference, let $z_i(x)=(b_i^*-\omega_i\cdot x)/\eps_M$.
Since $\omega_i\cdot x\le H_\Omega(\omega_i)$ on $\Omega$,
\eqref{eq:last-node} gives
\[
 z_i(x)>\frac{\Delta_M-h_M}{\eps_M}=a_DT_M-\frac{h_M}{\eps_M}.
\]
By \eqref{eq:parameters} and \eqref{eq:offset-spacing-choice},
$h_M/\eps_M\le C/T_M$, so $z_i(x)\to+\infty$ uniformly in $i$ and
$x\in\Omega$.
By evenness, the feature at $b_i^*+\ell h_M$ has value
$\sigma_D(z_i(x)+\ell h_M/\eps_M)$. Since
$\sigma_D(z)=4\sinh(2)e^{-2z}+O(e^{-4z})$ as $z\to+\infty$, this
shift multiplies the leading term by $e^{-2\ell h_M/\eps_M}$.
To retain the leading terms from the discarded nodes in
\eqref{eq:positive-grid-folding}, we add $w_i d_i(u)$ to the
coefficient at $b_i^*$, where
\begin{equation}\label{eq:correction}
 d_i(u)=
 \begin{cases}
  0,&\sigma=\sigma_G,\\[4pt]
  h_M\displaystyle\sum_{\ell=1}^\infty
  A_M[u](\omega_i,b_i^*+\ell h_M)e^{-2\ell h_M/\eps_M},
  &\sigma=\sigma_D
 \end{cases}.
\end{equation}
The series converges absolutely by \eqref{eq:fixed-resolution-density-bound},
and the uniform bound \eqref{eq:coefficient-envelope} gives
\begin{equation}\label{eq:correction-bound}
 \begin{aligned}
 |d_i(u)|
 &\le\frac{h_M}{e^{2h_M/\eps_M}-1}
            \sup_{b\in\R}|A_M[u](\omega_i,b)|\\
 &\le\frac{\eps_M}2\sup_{b\in\R}|A_M[u](\omega_i,b)|
 \le C\eps_M\Gamma_ME_{M,i}.
 \end{aligned}
\end{equation}
With the correction at the last node in each direction, the coefficients
of the retained features are
\begin{equation}\label{eq:finite-coefficients}
 c_{ij,M}(u)=w_ih_MA_M[u](\omega_i,b_{j,M}) +\mathbf 1_{\{j=K_{i,M}\}}w_id_i(u),\qquad
 (i,j)\in I_M^\Omega.
\end{equation}
At each retained node, the feature in \eqref{eq:ridge-features} satisfies
$\phi_{ij,M}=g_{\omega_i,b_{j,M}}$. With these coefficients, the finite
approximation is
\begin{equation}\label{eq:finite-operator}
 \mathcal A_M^\Omega u
 =\sum_{(i,j)\in I_M^\Omega}c_{ij,M}(u)\phi_{ij,M}.
\end{equation}

To estimate the discarded tails on $\Omega$, write
$g_i(x)=H_\Omega(\omega_i)-\omega_i\cdot x\ge0$. The distance from $x$
to a discarded ridge hyperplane is $b_{j,M}-\omega_i\cdot x$.
For $\delta>0$, an exponential in this distance factors as
\[
 e^{-(b_{j,M}-\omega_i\cdot x)/\delta}
 =e^{-(b_{j,M}-H_\Omega(\omega_i))/\delta}\,e^{-g_i(x)/\delta}.
\]
The second factor is common to all offsets in direction $\omega_i$.
Summing the offset bounds in that direction leaves a multiple
$v_i e^{-g_i(x)/\delta}$. We therefore need an $L^2(\Omega)$ bound
for $\sum_i w_i v_i e^{-g_i(x)/\delta}$ in terms of the coefficients $v_i$.

\begin{lemma}\label{lem:support-kernel}
For every $\delta>0$ and $v=(v_i)_{i=1}^{J_M}\in\mathbb C^{J_M}$,
\begin{equation}\label{eq:support-kernel}
 \begin{aligned}
 &\left\|\sum_iw_iv_i e^{-g_i(x)/\delta}
       \right\|_{L^2(\Omega)}\\
 &\quad\le C_{d,R_\Omega}\delta^{d/2}
       \left[1+(q_M\delta)^{-(d-1)}\right]^{1/2}
       \left(\sum_iw_i|v_i|^2\right)^{1/2},
 \end{aligned}
\end{equation}
where the constant is independent of $\delta$ and $v$.
\end{lemma}
The proof is given in Appendix~\ref{app:support-truncation}.

\begin{lemma}[Gaussian truncation]\label{lem:tail-G}
Let $\sigma=\sigma_G$ and fix $a_G>0$ with
$a_G^2\ge\max\{s-d/2,1/4\}$. Then, for all sufficiently large $M$,
\begin{equation*}
 \norm{\mathcal A_{M,q_M,h_M}^{\infty}u-\mathcal A_M^\Omega u}_{H^m(\Omega)}
 \le CM^{m-s}\norm u_{H^s(\Omega)}
\end{equation*}
for every $u\in H^s(\Omega)$ and every integer $0\le m\le m_*$.
\end{lemma}
\begin{proof}
Set $a_i=H_\Omega(\omega_i)+\Delta_M$.
The identity \eqref{eq:positive-grid-folding} gives
\begin{equation*}
 \begin{aligned}
 &(\mathcal A_{M,q_M,h_M}^{\infty}-\mathcal A_M^\Omega)u(x)\\
 &\quad=h_M\sum_iw_i\sum_{b_{j,M}>a_i}
 A_M[u](\omega_i,b_{j,M})\sigma_G\bigl((\omega_i\cdot x-b_{j,M})/\eps_M\bigr).
 \end{aligned}
\end{equation*}
At a discarded node $b_{j,M}>a_i$, the absolute value of the
Gaussian argument is
\[
 \frac{b_{j,M}-\omega_i\cdot x}{\eps_M}=\lambda+t,\qquad
 \lambda=\frac{\Delta_M}{\eps_M},\qquad
 t=\frac{b_{j,M}-a_i+g_i(x)}{\eps_M}\ge0.
\]
For every integer $k\ge0$, the Gaussian derivative satisfies
$|\sigma_G^{(k)}(z)|\le C_k(1+|z|)^k e^{-z^2}$.
Since $1+\lambda+t\le(1+\lambda)(1+t)$ and
$(1+t)^k e^{-t^2+t/2}$ is bounded for $t\ge0$, we have
\begin{equation*}
 (1+\lambda+t)^k e^{-(\lambda+t)^2}
 \le C_k(1+\lambda)^k e^{-\lambda^2}
                      e^{-(1+\lambda)t/2}.
\end{equation*}
The factor $e^{-\lambda^2}$ accounts for the margin. With
$\delta_G=\eps_M/(1+\lambda)$, the remaining exponential splits as
\[
 e^{-(1+\lambda)t/2}
 =e^{-(b_{j,M}-a_i)/(2\delta_G)}\,e^{-g_i(x)/(2\delta_G)}.
\]
Only the first factor varies with the offset. For $a\ge0$ and $\delta>0$,
summing the geometric series over the discarded nodes gives
\begin{equation}\label{eq:geometric-tail}
 h_M\sum_{b_{j,M}>a}e^{-(b_{j,M}-a)/\delta}
 \le\frac{h_M}{1-e^{-h_M/\delta}}
 =h_M+\frac{h_M}{e^{h_M/\delta}-1}\le h_M+\delta.
\end{equation}
For $|\alpha|=k$, the chain rule contributes a factor
$\eps_M^{-k}\omega_i^\alpha$, with $|\omega_i^\alpha|\le1$.
Using \eqref{eq:coefficient-envelope} for $A_M$ and
\eqref{eq:geometric-tail} with $\delta=2\delta_G$ for the sum over
discarded offsets, we obtain
\[
 \begin{aligned}
 &\left|\partial^\alpha (\mathcal A_{M,q_M,h_M}^{\infty}-\mathcal A_M^\Omega)u(x)\right|\\
 &\quad\le
 C_k\Gamma_M\eps_M^{-k} (1+\lambda)^k e^{-\lambda^2}(\delta_G+h_M)
 \sum_iw_iE_{M,i}e^{-g_i(x)/(2\delta_G)}.
 \end{aligned}
\]

Since $a_G$ is fixed, the scales in \eqref{eq:parameters} give
\[
 \lambda=a_G\sqrt{T_M},\qquad
 \delta_G=\frac{\sqrt{T_M}}{M(1+a_G\sqrt{T_M})}\asymp M^{-1},\qquad
 q_M\delta_G\ge c,\qquad h_M/\delta_G\le C.
\]
The $L^2(\Omega)$ norm of the sum over directions is bounded by
\eqref{eq:support-kernel} at scale $2\delta_G$ with $v_i=E_{M,i}$.
Together with \eqref{eq:sampled-energy}, the pointwise estimate gives,
for $|\alpha|=k$,
\[
 \begin{aligned}
 \norm{\partial^\alpha(\mathcal A_{M,q_M,h_M}^{\infty}
             -\mathcal A_M^\Omega)u}_{L^2(\Omega)}
 &\le C\Gamma_M\eps_M^{-k}(1+\lambda)^k
       e^{-\lambda^2}\delta_G^{1+d/2}\norm u_{H^s(\Omega)}\\
 &=C\Gamma_M\delta_G^{1+d/2-k}
       e^{-\lambda^2}\norm u_{H^s(\Omega)}.
 \end{aligned}
\]
Since $\delta_G\le1$ for large $M$, summing over $|\alpha|\le m$ gives
\begin{equation}\label{eq:gaussian-tail-strong}
 \norm{\mathcal A_{M,q_M,h_M}^{\infty}u-\mathcal A_M^\Omega u}_{H^m(\Omega)}
 \le C\Gamma_M\delta_G^{1+d/2-m}e^{-\lambda^2}\norm u_{H^s(\Omega)}.
\end{equation}
Using $\delta_G\asymp M^{-1}$, $e^{-\lambda^2}=(M+1)^{-a_G^2}$,
and the Gaussian estimate for $\Gamma_M$ in \eqref{eq:Gamma-estimates},
we bound the right-hand side of \eqref{eq:gaussian-tail-strong} by
\[
 CM^{m-s}\bigl(M^{s-d/2-a_G^2}T_M^{-1/2}
              +M^{1/4-a_G^2}T_M^{-1}\bigr)\norm u_{H^s(\Omega)}.
\]
The hypothesis on $a_G$ makes both exponents of $M$ in parentheses
nonpositive. Since $T_M\ge1$ for large $M$, the desired bound follows.
\end{proof}

At the $\ell$th node beyond $b_i^*$, the scaled shift is
$v=\ell h_M/\eps_M$. The tanh-difference correction pairs
$\sigma_D(z+v)$ with $e^{-2v}\sigma_D(z)$.
Their leading exponentials cancel. We estimate the difference and its
derivatives in $z$.

\begin{lemma}\label{lem:tanh-remainder}
Let $c_1=4\sinh2$ and $R_1(z)=\sigma_D(z)-c_1e^{-2z}$. For each
integer $k\ge0$,
\begin{equation}\label{eq:tanh-remainder}
 |R_1^{(k)}(z)|\le C_k e^{-4z},\qquad z\in\R.
\end{equation}
Consequently, with
$K_k(z,v)=\sigma_D^{(k)}(z+v)-e^{-2v}\sigma_D^{(k)}(z)$,
\begin{equation*}
 |K_k(z,v)|\le C_k e^{-4z}e^{-2v},\qquad z\in\R,\quad v\ge0.
\end{equation*}
\end{lemma}
The proof is given in Appendix~\ref{app:support-truncation}.

\begin{lemma}[Tanh-difference truncation]\label{lem:tail-D}
Let $\sigma=\sigma_D$ and fix $a_D>0$ with
$4a_D>\max\{s-d/2,(\pi+1)/2\}$. Then, for all sufficiently large $M$,
\begin{equation}\label{eq:tail-D}
 \norm{\mathcal A_{M,q_M,h_M}^{\infty}u-\mathcal A_M^\Omega u}_{H^m(\Omega)}
 \le CM^{m-s}\norm u_{H^s(\Omega)}
\end{equation}
for every $u\in H^s(\Omega)$ and every integer $0\le m\le m_*$.
\end{lemma}
\begin{proof}
Subtracting the corrected finite sum from \eqref{eq:positive-grid-folding}
and using \eqref{eq:correction} gives
\begin{equation*}
 \begin{split}
 (\mathcal A_{M,q_M,h_M}^{\infty}-\mathcal A_M^\Omega)u(x)
 ={}&h_M\sum_iw_i\sum_{\ell\ge1}A_M[u](\omega_i,b_i^*+\ell h_M)\\
 &\quad\times\bigl[\sigma_D(z_i(x)+\ell h_M/\eps_M)
                   -e^{-2\ell h_M/\eps_M}\sigma_D(z_i(x))\bigr].
 \end{split}
\end{equation*}
For $k=|\alpha|$, applying $\partial_x^\alpha$ to the bracket gives
$(-\omega_i)^\alpha\eps_M^{-k}K_k(z_i(x),\ell h_M/\eps_M)$.
The estimate for $K_k$ in \cref{lem:tanh-remainder} and the
coefficient bound \eqref{eq:coefficient-envelope} therefore give
\[
 |\partial^\alpha(\mathcal A_{M,q_M,h_M}^{\infty}-\mathcal A_M^\Omega)u(x)| \le
 C_k\eps_M^{-k}\Gamma_M \sum_iw_iE_{M,i}e^{-4z_i(x)} h_M\sum_{\ell\ge1}e^{-2\ell h_M/\eps_M}.
\]
Since
\begin{equation*}
 h_M\sum_{\ell\ge1}e^{-2\ell h_M/\eps_M}
 =\frac{h_M}{e^{2h_M/\eps_M}-1}\le\frac{\eps_M}2,
\end{equation*}
we obtain
\begin{equation}\label{eq:tanh-tail-pointwise}
 |\partial^\alpha(\mathcal A_{M,q_M,h_M}^{\infty}-\mathcal A_M^\Omega)u(x)|
 \le C_k\eps_M^{1-k}\Gamma_M
       \sum_iw_iE_{M,i}e^{-4z_i(x)}.
\end{equation}
The last-node bound \eqref{eq:last-node} gives
$z_i(x)>(g_i(x)+\Delta_M-h_M)/\eps_M$, and therefore
\begin{equation*}
 e^{-4z_i(x)}\le e^{4h_M/\eps_M}e^{-4\Delta_M/\eps_M}e^{-4g_i(x)/\eps_M}.
\end{equation*}
The factor $e^{4h_M/\eps_M}$ is uniformly bounded since
$h_M/\eps_M\le C/T_M$. At $\delta=\eps_M/4$, the factor involving
$q_M\delta$ in \eqref{eq:support-kernel} is also bounded because
$q_M\eps_M\ge C_{\rm ang}T_M$.
Applying that estimate with $v_i=E_{M,i}$ to
\eqref{eq:tanh-tail-pointwise}, followed by \eqref{eq:sampled-energy},
and summing over $|\alpha|\le m$ with $\eps_M\le1$, we obtain
\begin{equation*}
 \norm{\mathcal A_{M,q_M,h_M}^{\infty}u-\mathcal A_M^\Omega u}_{H^m(\Omega)}
 \le C\Gamma_M\eps_M^{1+d/2-m}
 e^{4h_M/\eps_M}e^{-4\Delta_M/\eps_M}\norm u_{H^s(\Omega)}.
\end{equation*}
With $\eps_M=T_M/M$ and $\Delta_M/\eps_M=a_DT_M$, the tanh-difference
estimate for $\Gamma_M$ in \eqref{eq:Gamma-estimates} gives the upper bound
\begin{equation*}
 CM^{m-s}e^{4h_M/\eps_M}e^{-4a_DT_M} \bigl(M^{s-d/2}T_M^{d/2-m}
 +M^{(\pi+1)/2}T_M^{d/2-m-1}\bigr)\norm u_{H^s(\Omega)}.
\end{equation*}
The hypothesis on $a_D$ gives $\gamma_D=4a_D-\max\{s-d/2,(\pi+1)/2\}>0$.
Since $e^{-4a_DT_M}=(M+1)^{-4a_D}$, multiplying either power of $M$
in parentheses by this factor gives at most $M^{-\gamma_D}$. For
$m\ge0$ and $T_M\ge1$, both logarithmic factors are at most $T_M^{d/2}$.
Thus $M^{-\gamma_D}T_M^{d/2}\to0$ as $M\to\infty$, which proves
\eqref{eq:tail-D}.
\end{proof}

\subsection{Proof of the approximation theorem}
\label{sec:proof-main}

\begin{proof}[Proof of \cref{thm:main}]
Choose the margin constants $a_G$ and $a_D$ as in
\cref{lem:tail-G,lem:tail-D}, respectively, and let $\mathcal A_M^\Omega$
denote the resulting operator in \eqref{eq:finite-operator}.
Fix $u\in H^s(\Omega)$ and an integer $0\le m\le m_*$.
The approximation error decomposes as
\begin{equation}\label{eq:four-errors}
 \begin{split}
 u-\mathcal A_M^\Omega u
 ={}&(u-f_M^\star)+(f_M^\star-\mathcal S_{M,q_M}u)\\
 &+(\mathcal S_{M,q_M}u-\mathcal A^\infty_{M,q_M,h_M}u)
 +(\mathcal A^\infty_{M,q_M,h_M}u-\mathcal A_M^\Omega u).
 \end{split}
\end{equation}
For sufficiently large $M$, \cref{lem:frequency}, the angular estimate
\eqref{eq:angular-target}, and \cref{lem:poisson} bound the first
three terms in \eqref{eq:four-errors}, and the
truncation estimate for the chosen activation bounds the last term.
The triangle inequality therefore gives
\[
 \norm{u-\mathcal A_M^\Omega u}_{H^m(\Omega)}
 \le CM^{m-s}\norm u_{H^s(\Omega)}.
\]
By \eqref{eq:sphere-count} and \cref{lem:poisson},
$J_M\asymp M^{d-1}$ and $h_M\asymp M^{-1}$. The feature count in
Section~\ref{subsec:space-construction} thus gives $N_{\sigma,M}\asymp M^d$,
so $M^{m-s}\asymp N_{\sigma,M}^{-(s-m)/d}$ as claimed in \eqref{eq:main-error}.

For the coefficient estimate, \eqref{eq:offset-nodes} and
\eqref{eq:sampled-energy} give
\begin{equation}\label{eq:coefficient-sum-inputs}
 h_MK_{i,M}\le R_\Omega+\Delta_M+h_M/2\le C,\qquad \sum_iw_iE_{M,i}^2\le C\norm
 u_{H^s(\Omega)}^2.
\end{equation}
The uniform bound \eqref{eq:coefficient-envelope} on the coefficient density gives
\[
 \sum_iw_ih_M\sum_{j=1}^{K_{i,M}}|A_M(\omega_i,b_{j,M})|
 \le C\Gamma_M\sum_iw_iE_{M,i}.
\]
Since $d_i(u)=0$ for $\sigma_G$, the bound
\eqref{eq:correction-bound} implies, for both activations,
\[
 \sum_iw_i|d_i(u)|\le C\eps_M\Gamma_M\sum_iw_iE_{M,i}.
\]
Since $\sum_iw_i=|\Sph|$ by \eqref{eq:sphere-exactness},
the weighted Cauchy--Schwarz inequality and
\eqref{eq:coefficient-sum-inputs} imply
$\sum_iw_iE_{M,i}\le C\norm u_{H^s(\Omega)}$.
As $\eps_M\le1$ for sufficiently large $M$, the coefficients in
\eqref{eq:finite-coefficients} satisfy
\begin{equation*}
 \sum_{(i,j)\in I_M^\Omega}|c_{ij,M}(u)|
 \le C\Gamma_M\norm u_{H^s(\Omega)}.
\end{equation*}
Combining this estimate with \eqref{eq:Gamma-estimates} and
$T_M\ge1$ gives \eqref{eq:main-budget}, with the exponents in
\eqref{eq:coefficient-exponents}.

The coefficient formulas are linear in $u$, and the preceding
estimate shows that each $c_{ij,M}$ is a bounded functional on
$H^s(\Omega)$. The coefficients are real for real $u$ by
\cref{prop:representation} and \eqref{eq:correction}.
Thus \eqref{eq:finite-operator} defines a bounded linear map into
$V_M^\Omega$.
\end{proof}

\section{Regularized least-squares approximation}\label{sec:regularized-ls}

The coefficient bound in \cref{thm:main} gives a sufficient TSVD
threshold for preserving the optimal Sobolev approximation rate.
We apply the comparison estimate for regularized approximation
\cite{AdcockHuybrechs2019} first to the original ridge features and
then to the normalized features.

Fix an activation and the features constructed in
Section~\ref{subsec:space-construction}. We omit the activation subscript
and enumerate the features as $\phi_{1,M},\ldots,\phi_{N_M,M}$.
For an integer $0\le m\le m_*$, write $(\cdot,\cdot)_m$ for the
$H^m(\Omega)$ inner product and define the synthesis operator
\[
 \mathcal T_{M,m}:\C^{N_M}\longrightarrow H^m(\Omega),\qquad
 \mathcal T_{M,m}c=\sum_{\ell=1}^{N_M}c_\ell\phi_{\ell,M}.
\]
The continuous least-squares problem minimizes
$\norm{u-\mathcal T_{M,m}c}_{H^m(\Omega)}$ over $c\in\C^{N_M}$.

\Needspace{8\baselineskip}
Let $\lambda_1\ge\cdots\ge\lambda_{r_M}>0$ be the positive
eigenvalues of the Gram matrix $\mathcal T_{M,m}^*\mathcal T_{M,m}$,
where $r_M=\dim V_M^\Omega$. For corresponding orthonormal
eigenvectors $v_k$, set $z_k=\lambda_k^{-1/2}\mathcal T_{M,m}v_k$.
The functions $z_k$ are orthonormal in $H^m(\Omega)$.
For $0<\tau<1$, TSVD with relative singular-value threshold $\tau$
retains the indices $k$ with $\lambda_k>\tau^2\lambda_1$.
The resulting coefficient vector is
\begin{equation}
	c_{M,m}^{\tau}(u)
	=\sum_{\lambda_k>\tau^2\lambda_1}
	\frac{(u,z_k)_m}{\sqrt{\lambda_k}}v_k,
	\label{eq:continuous-tsvd-coefficients}
\end{equation}
and applying the synthesis operator $\mathcal T_{M,m}$ to this vector
gives the approximation
\begin{equation}
	 \mathcal R_{M,m}^{\tau}u
	=\mathcal T_{M,m}c_{M,m}^{\tau}(u)
	=\sum_{\lambda_k>\tau^2\lambda_1}(u,z_k)_m z_k.
	\label{eq:continuous-tsvd}
\end{equation}

\begin{theorem}\label{thm:continuous-tsvd}
Let $P_\sigma$ be given by \eqref{eq:coefficient-exponents}.
For all sufficiently large $M$, every $u\in H^s(\Omega)$,
every integer $0\le m\le m_*$, and every $0<\tau<1$,
\begin{equation}\label{eq:tsvd-sobolev-error}
 \norm{u-\mathcal R_{M,m}^{\tau}u}_{H^m(\Omega)} \le C\left(M^{m-s}
 +\tau\sqrt{N_M}\eps_M^{1/2-m}M^{P_\sigma}\right) \norm u_{H^s(\Omega)}.
\end{equation}
The constant is independent of $\tau$.
In particular, fix $\tau_0>0$ and choose $0<\tau_M<1$ such that
\begin{equation}\label{eq:tsvd-relative-threshold}
 \tau_M\le\tau_0
 \begin{cases}
 M^{-(s+P_G+(d-1)/2)}T_M^{-1/4},&\sigma=\sigma_G,\\[3pt]
 M^{-(s+P_D+(d-1)/2)}T_M^{-1/2},&\sigma=\sigma_D.
 \end{cases}
\end{equation}
Then
\begin{equation}\label{eq:tsvd-optimal-rate}
 \norm{u-\mathcal R_{M,m}^{\tau_M}u}_{H^m(\Omega)} \le C M^{m-s}\norm u_{H^s(\Omega)} \le
 C'N_M^{-(s-m)/d}\norm u_{H^s(\Omega)},
\end{equation}
with constants that may also depend on $\tau_0$.
\end{theorem}

\begin{proof}
By \eqref{eq:continuous-tsvd}, $\mathcal R_{M,m}^{\tau}$ is an
orthogonal projection, so it and its complement have norm at most
one. Orthogonality gives, for every $c\in\C^{N_M}$,
\[
 \norm{(I-\mathcal R_{M,m}^{\tau})\mathcal T_{M,m}c}_{H^m(\Omega)}^2
 =\sum_{\lambda_k\le\tau^2\lambda_1}
       \lambda_k|\langle c,v_k\rangle_{\ell^2}|^2
 \le\tau^2\lambda_1\norm c_2^2.
\]
For every $u\in H^m(\Omega)$, the triangle inequality,
contractivity of $I-\mathcal R_{M,m}^{\tau}$, and the preceding
bound give
\begin{equation}\label{eq:tsvd-comparison}
 \norm{u-\mathcal R_{M,m}^{\tau}u}_{H^m(\Omega)}
 \le\inf_{c\in\C^{N_M}}\left\{
 \norm{u-\mathcal T_{M,m}c}_{H^m(\Omega)}
 +\tau\sqrt{\lambda_1}\norm c_2\right\}.
\end{equation}

We next bound $\sqrt{\lambda_1}=\norm{\mathcal T_{M,m}}$.
For a feature $\phi_{\ell,M}(x)=\sigma((\omega\cdot x-b)/\eps_M)$,
the sections of $\Omega\subset B_{R_\Omega}$ perpendicular to the
unit vector $\omega$ have uniformly bounded $(d-1)$-dimensional
measure. The chain rule and integration over these sections give
\[
 \begin{aligned}
 \norm{\partial^\alpha\phi_{\ell,M}}_{L^2(\Omega)}^2
 &\le C\eps_M^{-2|\alpha|}
       \int_\R\left|
       \sigma^{(|\alpha|)}((t-b)/\eps_M)\right|^2\dd t\\
 &=C\eps_M^{1-2|\alpha|}
       \norm{\sigma^{(|\alpha|)}}_{L^2(\R)}^2.
 \end{aligned}
\]
Both activations and all their derivatives belong to $L^2(\R)$.
Since $\eps_M\le1$ for sufficiently large $M$, summing over
$|\alpha|\le m$ yields
\begin{equation}\label{eq:feature-norm-upper}
 \max_{1\le\ell\le N_M}
 \norm{\phi_{\ell,M}}_{H^m(\Omega)}
 \le C\eps_M^{1/2-m}.
\end{equation}
Consequently,
\[
 \sqrt{\lambda_1}
 \le\sqrt{N_M}\max_\ell\norm{\phi_{\ell,M}}_{H^m(\Omega)}
 \le C\sqrt{N_M}\eps_M^{1/2-m}.
\]
For $u\in H^s(\Omega)$, let $c_M(u)$ be the coefficient vector of
$\mathcal A_{\sigma,M}^\Omega u$ from \cref{thm:main}.
By \eqref{eq:main-budget},
\[
 \norm{c_M(u)}_2\le\norm{c_M(u)}_1
 \le CM^{P_\sigma}\norm u_{H^s(\Omega)}.
\]
Using this vector in \eqref{eq:tsvd-comparison}, together with
\eqref{eq:main-error}, proves \eqref{eq:tsvd-sobolev-error}.

Finally, $N_M\asymp M^d$ and the scales in \eqref{eq:parameters}
show that \eqref{eq:tsvd-relative-threshold} implies
\[
 \tau_M\sqrt{N_M}\eps_M^{1/2-m}M^{P_\sigma}
 \le C M^{m-s}
 \begin{cases}
 T_M^{-m/2},&\sigma=\sigma_G,\\
 T_M^{-m},&\sigma=\sigma_D.
 \end{cases}
\]
Since $T_M\ge1$ for sufficiently large $M$, this term is bounded by
$C M^{m-s}$. Substituting into \eqref{eq:tsvd-sobolev-error} and
using $N_M\asymp M^d$ gives \eqref{eq:tsvd-optimal-rate}.
\end{proof}

We now normalize the features in $H^m(\Omega)$. Define the positive
diagonal matrix $D_{M,m}$ and the normalized synthesis operator by
\[
 D_{M,m}=\operatorname{diag}\bigl(
 \norm{\phi_{\ell,M}}_{H^m(\Omega)}\bigr)_{\ell=1}^{N_M},\qquad
 \mathcal B_{M,m}=\mathcal T_{M,m}D_{M,m}^{-1}.
\]
Let $\widetilde{\mathcal R}_{M,m}^{\tau}$ be the TSVD approximation
obtained by applying the same relative singular-value threshold
$\tau$ to $\mathcal B_{M,m}$.

\begin{corollary}\label{cor:normalized-tsvd}
For all sufficiently large $M$, every $u\in H^s(\Omega)$,
every integer $0\le m\le m_*$, and every $0<\tau<1$,
the estimate \eqref{eq:tsvd-sobolev-error} holds with
$\mathcal R_{M,m}^{\tau}$ replaced by
$\widetilde{\mathcal R}_{M,m}^{\tau}$.
If $\tau_M$ satisfies \eqref{eq:tsvd-relative-threshold}, then
$\widetilde{\mathcal R}_{M,m}^{\tau_M}$ also satisfies
\eqref{eq:tsvd-optimal-rate}, with the same dependence of the
constants as in \cref{thm:continuous-tsvd}.
\end{corollary}
\begin{proof}
Column normalization gives $\norm{\mathcal B_{M,m}}\le\sqrt{N_M}$.
For the coefficient vector $c_M(u)$ above,
\eqref{eq:feature-norm-upper} and \eqref{eq:main-budget} give
\[
 \norm{D_{M,m}c_M(u)}_2
 \le\left(\max_\ell\norm{\phi_{\ell,M}}_{H^m(\Omega)}\right)
       \norm{c_M(u)}_2
 \le C\eps_M^{1/2-m}M^{P_\sigma}\norm u_{H^s(\Omega)}.
\]
Apply the comparison argument in \eqref{eq:tsvd-comparison} to
$\mathcal B_{M,m}$ and the normalized coefficient vector
$a=D_{M,m}c_M(u)$. Since
$\mathcal B_{M,m}a=\mathcal A_{\sigma,M}^\Omega u$,
\eqref{eq:main-error} gives \eqref{eq:tsvd-sobolev-error} for
$\widetilde{\mathcal R}_{M,m}^{\tau}$.
The same threshold calculation gives \eqref{eq:tsvd-optimal-rate}.
\end{proof}

The threshold is independent of $m$, whereas the Sobolev projections
and the normalization depend on $m$.

\begingroup
\raggedbottom
\widowpenalty=10000
\setlength{\floatsep}{8pt plus 1pt minus 1pt}
\setlength{\textfloatsep}{10pt plus 1pt minus 1pt}
\setlength{\intextsep}{4pt plus 1pt minus 1pt}
\section{Numerical experiments}\label{sec:numerics}

We examine $L^2$ and $H^1$ approximation on $\Omega=(-1,1)^d$, $d=2,3$,
using the random Fourier benchmarks of Fu and Wang~\cite{fu2026optimal}.

\subsection{Experimental setup}
Let $\mathcal F_{32}$ consist of the nonzero integer vectors $n$ with
$\max_j|n_j|\le32$ whose first nonzero component is positive.
For $n\in\mathcal F_{32}$, take independent standard Gaussian coefficients
$a_n,b_n$ and define the benchmark on $[0,1]^d$ by
\begin{equation*}
 \widetilde g_s(y)=
 \sum_{n\in\mathcal F_{32}}
 \frac{a_n\cos(2\pi n\cdot y)+b_n\sin(2\pi n\cdot y)}
 {(1+|n|^2)^{(s+d/2)/2}\log(2+|n|)}.
\end{equation*}
Set $Z_s=\norm{\widetilde g_s}_{L^2([0,1]^d)}$ and
$\boldsymbol1=(1,\ldots,1)$. The normalized target on $\Omega$ is
\begin{equation*}
 u_s(x)=2^{-d/2}Z_s^{-1}
 \widetilde g_s\bigl((x+\boldsymbol1)/2\bigr),
 \qquad \norm{u_s}_{L^2(\Omega)}=1.
\end{equation*}
We take $s=2,4,6$ and keep each target fixed across resolutions and
activations. These finite Fourier sums are smooth, and $s$ controls
their spectral decay.

For $m=0,1$, let $v_{M,m}^{\rm LS}$ be the weighted $H^m$ least-squares
approximation in the space \eqref{eq:space}, computed separately for
each norm. The $H^1$ fit includes function values and all first spatial
derivatives. We use tensor-product Gauss--Legendre quadrature with
$129$ nodes per coordinate for $d=2$ and $64$ for $d=3$.
We normalize the columns of the weighted system to unit Euclidean
norm and solve it by truncated singular value decomposition with
a relative singular-value threshold of $10^{-15}$.

We measure the relative error by
\begin{equation*}
 E_m=\frac{\norm{u_s-v_{M,m}^{\rm LS}}_{H^m(\Omega)}}
 {\norm{u_s}_{H^m(\Omega)}},\qquad m=0,1.
\end{equation*}
We approximate the numerator using the same quadrature rule as in
the least-squares fit and evaluate the denominator exactly from the
Fourier coefficients by Parseval's identity.

The curves in \cref{fig:results-2d,fig:results-3d} show the mean $\overline E_m$ over
the random targets, with shaded bands extending one sample standard
deviation above and below the mean.
Dashed segments indicate the reference decay $N_M^{-(s-m)/d}$.
The fitted order $p$ is the negative slope of a linear fit of
$\log\overline E_m$ against $\log N_M$ over the last four resolutions.

\Needspace{12\baselineskip}
\subsection{Two-dimensional results}
Figure~\ref{fig:results-2d} shows Gaussian $L^2$ and tanh-difference
$H^1$ approximation. The mean errors decrease with $N_M$ and decay
faster as $s$ increases. The fitted orders in
\cref{tab:numerical-orders-2d} exceed the reference values $s/2$ and
$(s-1)/2$, respectively, over the tested range.

\begin{table}[H]
\caption{Mean relative errors $\overline E_m$ at the finest tested
resolution and fitted orders $p$ for $d=2$.}
\label{tab:numerical-orders-2d}
\centering
\begingroup
\small
\setlength{\tabcolsep}{4pt}
\renewcommand{\arraystretch}{1.05}
\begin{tabular*}{\linewidth}{@{\extracolsep{\fill}}l*{3}{cc}@{}}
\toprule
\multirow{2}{*}{Norm, activation}
& \multicolumn{2}{c}{$s=2$} & \multicolumn{2}{c}{$s=4$}
& \multicolumn{2}{c}{$s=6$} \\
\cmidrule(lr){2-3}\cmidrule(lr){4-5}\cmidrule(l){6-7}
& $\overline E_m$ & $p$
& $\overline E_m$ & $p$
& $\overline E_m$ & $p$ \\
\midrule
$L^2$, $\sigma_G$ & $1.598\times10^{-3}$ & 1.173
& $1.128\times10^{-5}$ & 2.131
& $9.711\times10^{-8}$ & 3.118 \\
$H^1$, $\sigma_D$ & $2.363\times10^{-2}$ & 0.750
& $1.707\times10^{-4}$ & 1.655
& $1.479\times10^{-6}$ & 2.586 \\
\bottomrule
\end{tabular*}
\endgroup
\end{table}

\begin{figure}[H]
 \centering
 \begin{minipage}[t]{0.49\linewidth}
  \centering
  \includegraphics[width=\linewidth]{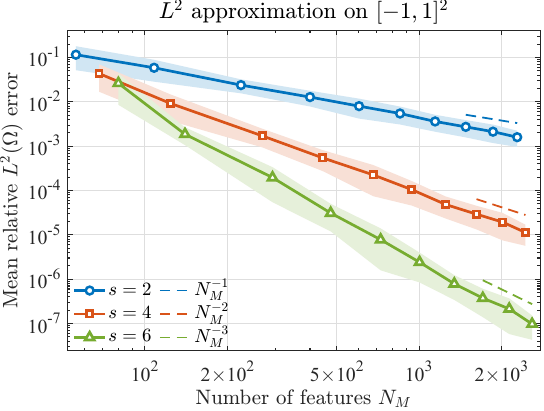}\\[2pt]
  {\small (a) Gaussian activation}
 \end{minipage}\hfill
 \begin{minipage}[t]{0.49\linewidth}
  \centering
  \includegraphics[width=\linewidth]{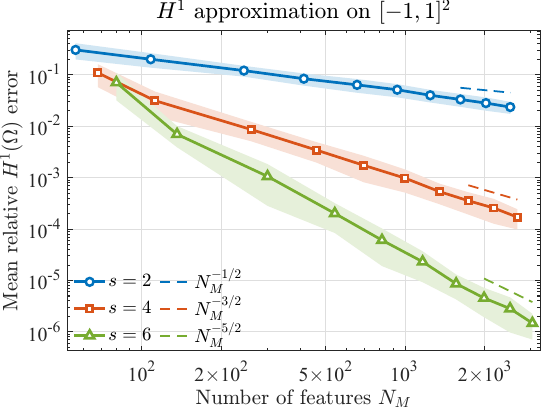}\\[2pt]
  {\small (b) Tanh-difference activation}
 \end{minipage}
 \caption{Gaussian $L^2$ approximation (left) and tanh-difference
 $H^1$ approximation (right) on $[-1,1]^2$, using five random
 targets for each $s$.}
 \label{fig:results-2d}
\end{figure}

\Needspace{12\baselineskip}
\subsection{Three-dimensional results}
Figure~\ref{fig:results-3d} presents tanh-difference $L^2$ and Gaussian
$H^1$ approximation. The mean errors again decrease faster as $s$
increases. The fitted orders in \cref{tab:numerical-orders-3d}
exceed $s/3$ for the tanh difference and $(s-1)/3$ for the Gaussian
over the tested range.

\begin{table}[H]
\caption{Mean relative errors $\overline E_m$ at the finest tested
resolution and fitted orders $p$ for $d=3$.}
\label{tab:numerical-orders-3d}
\centering
\begingroup
\small
\setlength{\tabcolsep}{4pt}
\renewcommand{\arraystretch}{1.05}
\begin{tabular*}{\linewidth}{@{\extracolsep{\fill}}l*{3}{cc}@{}}
\toprule
\multirow{2}{*}{Norm, activation}
& \multicolumn{2}{c}{$s=2$} & \multicolumn{2}{c}{$s=4$}
& \multicolumn{2}{c}{$s=6$} \\
\cmidrule(lr){2-3}\cmidrule(lr){4-5}\cmidrule(l){6-7}
& $\overline E_m$ & $p$
& $\overline E_m$ & $p$
& $\overline E_m$ & $p$ \\
\midrule
$L^2$, $\sigma_D$ & $8.006\times10^{-3}$ & 0.756
& $1.973\times10^{-4}$ & 1.476
& $5.929\times10^{-6}$ & 2.257 \\
$H^1$, $\sigma_G$ & $7.382\times10^{-2}$ & 0.457
& $1.792\times10^{-3}$ & 1.196
& $5.404\times10^{-5}$ & 2.037 \\
\bottomrule
\end{tabular*}
\endgroup
\end{table}

\begin{figure}[H]
 \centering
 \begin{minipage}[t]{0.49\linewidth}
  \centering
  \includegraphics[width=\linewidth]{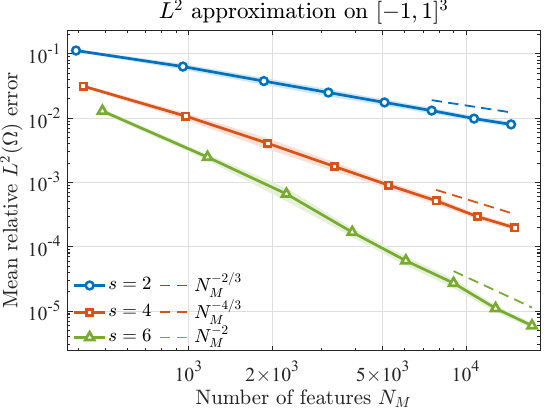}\\[2pt]
  {\small (a) Tanh-difference activation}
 \end{minipage}\hfill
 \begin{minipage}[t]{0.49\linewidth}
  \centering
  \includegraphics[width=\linewidth]{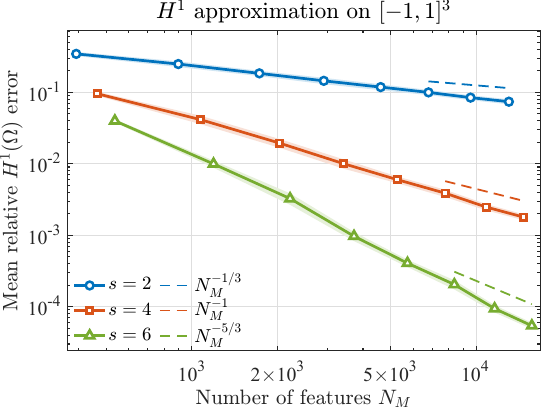}\\[2pt]
  {\small (b) Gaussian activation}
 \end{minipage}
 \caption{Tanh-difference $L^2$ approximation (left) and Gaussian
 $H^1$ approximation (right) on $[-1,1]^3$, using three random
 targets for each $s$.}
 \label{fig:results-3d}
\end{figure}

\FloatBarrier
\endgroup

\section{Conclusion}\label{sec:conclusion}
We have constructed explicit, target-independent shallow-network spaces
for Gaussian and tanh-difference activations on bounded Lipschitz
domains. With $N_M\asymp M^d$ features, the associated linear operators
attain the optimal fixed-space rate $N_M^{-(s-m)/d}$ simultaneously for
all integer orders $0\le m<s$. The direction-dependent offset cutoffs have
vanishing exterior margins, and the output coefficients satisfy
polynomial bounds in the frequency resolution $M$.
These coefficient bounds also control the continuous Sobolev
approximation error of TSVD with either the original or normalized
features. An explicit relative threshold preserves the optimal rate
for every integer order $0\le m<s$ and can be chosen independently
of $m$.

The numerical experiments use weighted least-squares fits in the
truncated spaces for Fourier targets in two and three dimensions.
For both activations, the mean errors decrease over the tested
frequency resolutions, and the fitted orders increase with the
regularity index $s$.

Sharper approximation and quadrature estimates could reduce the
conservative theoretical node counts while preserving the proved rate.
Adaptive selection based on the target or residual could also refine
the directions and offsets. Condition number estimates for the Sobolev
Gram matrices of the features spanning $V_M^\Omega$ are also important for
numerical stability. We will investigate these issues in future work.

\appendix
\section{Estimates for \texorpdfstring{$\Gamma_M$}{Gamma M}}\label{app:coefficient-bounds}

\begin{proof}[Proof of \cref{lem:budgets}]
\emph{Gaussian activation.} Here $\eps_M=\sqrt{T_M}/M$ and
$|\mathfrak r_M(\tau)|^2=\pi^{-1}e^{T_M\tau^2/(2M^2)}$. Hence
\begin{equation*}
 \Gamma_M^2=\frac{M^2}{\pi T_M}
 \int_0^M\tau^{d-1}\ip\tau^{-2s}e^{T_M\tau^2/(2M^2)}\dd\tau.
\end{equation*}
Split the integral at $M/\sqrt{T_M}$. On the lower interval, the
exponential is bounded, so
\begin{equation*}
 \int_0^{M/\sqrt{T_M}}\tau^{d-1}\ip\tau^{-2s} e^{T_M\tau^2/(2M^2)}\dd\tau \le C \begin{cases}
 (M/\sqrt{T_M})^{d-2s},&s<d/2,\\ \log(M/\sqrt{T_M}),&s=d/2,\\ 1,&s>d/2. \end{cases}
\end{equation*}
On the upper interval, use $\ip\tau^{-2s}\le\tau^{-2s}$ and
substitute $v=\sqrt{T_M}\tau/M$ to obtain
\begin{equation}\label{eq:gaussian-budget-endpoint}
 \begin{aligned}
 \int_{M/\sqrt{T_M}}^M\tau^{d-1-2s}e^{T_M\tau^2/(2M^2)}\dd\tau
 &=\left(\frac M{\sqrt{T_M}}\right)^{d-2s}
          \int_1^{\sqrt{T_M}}v^{d-1-2s}e^{v^2/2}\dd v\\
 &\le C\frac{M^{d-2s}e^{T_M/2}}{T_M}.
 \end{aligned}
\end{equation}
For the last inequality, observe that
\[
 \frac{\dd}{\dd v}\bigl(v^{d-2-2s}e^{v^2/2}\bigr)
 =v^{d-1-2s}e^{v^2/2}\left(1+\frac{d-2-2s}{v^2}\right) \ge\frac12v^{d-1-2s}e^{v^2/2}
\]
for sufficiently large $v$. Integrating this inequality and absorbing
the integral over the fixed initial interval proves the bound.
For $s<d/2$, the ratio of the upper bound in
\eqref{eq:gaussian-budget-endpoint} to $(M/\sqrt{T_M})^{d-2s}$ is
$e^{T_M/2}T_M^{d/2-s-1}\to\infty$. At $s=d/2$, use
$\log(M/\sqrt{T_M})\le T_M$ and $e^{T_M/2}/T_M^2\to\infty$ to absorb
the logarithmic contribution.
For $s>d/2$, retain the bounded lower-interval contribution.
Using $e^{T_M/2}=(M+1)^{1/2}$, we conclude that
\begin{equation*}
 \Gamma_M^2
 \le C\left(\frac{M^2}{T_M}
       +\frac{M^{d+5/2-2s}}{T_M^2}\right).
\end{equation*}
Taking square roots gives the Gaussian bound in
\eqref{eq:Gamma-estimates}.

\emph{Tanh-difference activation.} Here $\eps_M=T_M/M$.
Split the defining integral at $2M/T_M$.
Since $F_D$ has a bounded reciprocal on $[0,2]$,
\begin{equation*}
 \eps_M^{-2}\int_0^{2M/T_M}\tau^{d-1}\ip\tau^{-2s} |\mathfrak r_M(\tau)|^2\dd\tau \le
 C\frac{M^2}{T_M^2} \begin{cases} (M/T_M)^{d-2s},&s<d/2,\\ \log(2+M/T_M),&s=d/2,\\ 1,&s>d/2.
 \end{cases}
\end{equation*}
On the remaining interval, the scaled frequency $\eta=T_M\tau/M$
runs over $[2,T_M]$. On each interval $\mathcal I_k$, write $v=\eta-k\pi$.
The exclusion width in \eqref{eq:scaled-notch-widths} and the bound
$|\sin v|\ge2|v|/\pi$ for $|v|\le\pi/2$ give
\begin{equation}\label{eq:tanh-root-cell-integral}
 \int_{\mathcal I_k}\frac{|\vartheta_M(\eta)|^2}{\sin^2\eta}\dd\eta \le
 C\int_{\delta_{k,M}}^{\pi/2}v^{-2}\dd v \le C\delta_{k,M}^{-1} \le CMT_M^{2s-1}k^{-2s}.
\end{equation}
On $\mathcal I_k$, we also have $\eta\asymp k$ and
$e^{\pi\eta}\le Ce^{\pi^2k}$.
The exponential sum estimate in the proof of
\cref{lem:inverse-multiplier}, with $\alpha=d-1-4s$ and $b=\pi^2$, gives
\[
 \sum_{k=1}^{K_M^{\rm root}}k^{d-1-4s}e^{\pi^2k}
 \le C(K_M^{\rm root})^{d-1-4s}e^{\pi^2K_M^{\rm root}}.
\]
The intervals $\mathcal I_k$ cover $[2,T_M]$, so we may bound the
nonnegative integral by the sum over the full intervals.
Using the formula for $F_D$, substituting $\eta=T_M\tau/M$, and
applying \eqref{eq:tanh-root-cell-integral} and the sum estimate above,
we obtain
\begin{equation}\label{eq:tanh-root-sum}
 \begin{aligned}
 &\eps_M^{-2}\int_{2M/T_M}^M\tau^{d-1}\ip\tau^{-2s}
                             |\mathfrak r_M(\tau)|^2\dd\tau\\
 &\quad\le C\left(\frac M{T_M}\right)^{d+2-2s}
       \int_2^{T_M}\eta^{d-1-2s}e^{\pi\eta}
               \frac{|\vartheta_M(\eta)|^2}{\sin^2\eta}\dd\eta\\
 &\quad\le C\left(\frac M{T_M}\right)^{d+2-2s}MT_M^{2s-1}
       \sum_{k=1}^{K_M^{\rm root}}k^{d-1-4s}e^{\pi^2k}\\
 &\quad\le C\left(\frac M{T_M}\right)^{d+2-2s}MT_M^{2s-1}
       T_M^{d-1-4s}(M+1)^\pi
 \le C\frac{M^{d+3+\pi-2s}}{T_M^4}.
 \end{aligned}
\end{equation}
Here we used $K_M^{\rm root}\asymp T_M$ and
$\pi K_M^{\rm root}\le T_M+\pi/2$ to estimate the sum over $k$.

For $s<d/2$, the ratio of the final bound in
\eqref{eq:tanh-root-sum} to the lower-interval bound
$(M/T_M)^{d+2-2s}$ is $M^{1+\pi}T_M^{d-2-2s}\to\infty$.
At $s=d/2$, the inequality $\log(2+M/T_M)\le CT_M$ bounds the
lower-interval contribution by $CM^2/T_M$, which is absorbed because
$M^{1+\pi}/T_M^3\to\infty$.
For $s>d/2$, retain $CM^2/T_M^2$. Combining the two intervals gives
\begin{equation*}
 \Gamma_M^2
 \le C\left(\frac{M^2}{T_M^2}
       +\frac{M^{d+3+\pi-2s}}{T_M^4}\right).
\end{equation*}
Taking square roots proves the tanh-difference bound in
\eqref{eq:Gamma-estimates}.
\end{proof}

\Needspace{8\baselineskip}
\section{Proofs of the discrete Fourier estimates}\label{app:fourier-estimates}
We prove \cref{lem:sampled-energy,lem:shifted-synthesis} using a
spherical cap estimate and a weighted Fourier estimate.

\begin{lemma}\label{lem:cap}
Let $q_M\ge1$ and let $\mu_{q_M}=\sum_iw_i\delta_{\omega_i}$ be the
measure associated with the cubature rule \eqref{eq:sphere-exactness},
where $\delta_{\omega_i}$ is the Dirac mass at $\omega_i$.
For $\theta\in\Sph$ and $\rho\ge0$,
let $\operatorname{cap}(\theta,\rho)$ be the closed geodesic cap
with center $\theta$ and radius $\rho$. Then
\begin{equation}\label{eq:cap-mass}
 \mu_{q_M}(\operatorname{cap}(\theta,\rho))
 \le C_d(\rho+q_M^{-1})^{d-1}.
\end{equation}
The constant $C_d$ depends only on $d$.
\end{lemma}
\begin{proof}
Let $\ell_0=\lceil d/2\rceil$, so $2\ell_0>d-1$.
For sufficiently large $q_M$, choose the largest odd integer
$n\le q_M/\ell_0+1$. Then $n\asymp q_M$ and
$\ell_0(n-1)\le q_M$, uniformly in $q_M$.
Writing $n=2k+1$, the identity
\[
 \frac{\sin(n\alpha/2)}{\sin(\alpha/2)}
 =1+2\sum_{j=1}^{k}\cos(j\alpha)
\]
shows that
\[
 P_n(\cos\alpha)
 =\left(\frac{\sin(n\alpha/2)}{n\sin(\alpha/2)}\right)^{2\ell_0}
\]
is a nonnegative polynomial in $\cos\alpha$ of degree
$\ell_0(n-1)\le q_M$. Its value at $\alpha=0$ is defined by continuity.
For $0\le\alpha\le1/n$, elementary sine inequalities give
$P_n(\cos\alpha)\ge(2/\pi)^{2\ell_0}$.
For $0\le\alpha\le\pi$,
\[
 0\le P_n(\cos\alpha)
 \le C\min\{1,(n\alpha)^{-2\ell_0}\}.
\]
Spherical polar coordinates around $\theta$ therefore imply
\[
 \begin{aligned}
 \int_{\Sph}P_n(\theta\cdot\omega)\dd S(\omega)
 &\le C\int_0^\pi
       \min\{1,(n\alpha)^{-2\ell_0}\}\alpha^{d-2}\dd\alpha\\
 &\le C\left[\int_0^{1/n}\alpha^{d-2}\dd\alpha
 +n^{-2\ell_0}\int_{1/n}^{\pi}\alpha^{d-2-2\ell_0}\dd\alpha\right]\\
 &\le Cn^{-(d-1)}.
 \end{aligned}
\]
Positivity and cubature exactness yield
\[
 \begin{aligned}
 (2/\pi)^{2\ell_0}\mu_{q_M}(\operatorname{cap}(\theta,1/n))
 &\le\int_{\Sph}P_n(\theta\cdot\omega)\dd\mu_{q_M}(\omega)\\
 &=\int_{\Sph}P_n(\theta\cdot\omega)\dd S(\omega)
 \le Cq_M^{-(d-1)}.
 \end{aligned}
\]
For $0\le\rho\le1$, choose a maximal $1/n$-separated set
$\mathcal C$ in $\operatorname{cap}(\theta,\rho)$, with separation
measured by geodesic distance. The caps of radius $1/n$ centered at
$\mathcal C$ cover the original cap by maximality. The caps of radius
$1/(2n)$ have disjoint interiors and lie in
$\operatorname{cap}(\theta,\rho+1/(2n))$. Since the surface area of a
cap of radius $r\le2$ is comparable to $r^{d-1}$,
\[
 c_dn^{-(d-1)}\#\mathcal C \le C_d(\rho+1/n)^{d-1},\qquad \#\mathcal C\le
 C_d(1+n\rho)^{d-1} \le C_d(1+q_M\rho)^{d-1}.
\]
Summing the mass bound for caps of radius $1/n$ over this cover proves
\eqref{eq:cap-mass} for $\rho\le1$. For $\rho>1$, and for the remaining
bounded values of $q_M$, the identity $\mu_{q_M}(\Sph)=|\Sph|$ gives the
result after increasing the constant.
\end{proof}

\Needspace{8\baselineskip}
For the frequency measures used below, the cap bound will control
the mass of unit balls. The next lemma converts this mass bound into
a weighted Fourier estimate for compactly supported functions.

\begin{lemma}\label{lem:fourier-measure}
Let $\beta\ge0$, $A>0$, and let $\mu$ be a positive Borel measure on $\R^d$ such that
\[
 D_\mu:=\sup_{\zeta\in\R^d}\mu(B(\zeta,1))<\infty.
\]
If $g\in H^\beta(\R^d)$ and $\supp g\subset B_A$, then
\begin{equation}\label{eq:fourier-measure-bound}
 \int_{\R^d}\ip\xi^{2\beta}|\widehat g(\xi)|^2\dd\mu(\xi)
 \le C_{d,\beta,A}D_\mu\norm g_{H^\beta(\R^d)}^2.
\end{equation}
Here $\widehat g$ denotes its continuous representative, and the
constant is independent of $\mu$ and $g$.
\end{lemma}
\begin{proof}
Choose a fixed $\chi\in C_c^\infty(\R^d)$ equal to one on $B_A$.
Compact support and $g\in L^2$ imply $g\in L^1$, so $\widehat g$
is continuous.
Since $\chi g=g$ and
$\norm g_{L^1}\norm{\widehat\chi}_{L^1}<\infty$,
Fourier inversion for $\chi$ and Fubini's theorem give, for every
$\xi\in\R^d$,
\[
 \widehat g(\xi)=(2\pi)^{-d}\int_{\R^d}
       \widehat\chi(\xi-\eta)\widehat g(\eta)\dd\eta.
\]
The weight inequality
$\ip\xi^\beta\le C_\beta\ip{\xi-\eta}^\beta\ip\eta^\beta$
and weighted Cauchy--Schwarz, with
$K(v)=|\widehat\chi(v)|\ip v^\beta\in L^1$, give
\[
 \begin{aligned}
 \ip\xi^{2\beta}|\widehat g(\xi)|^2
 &\le C\norm K_{L^1}
       \int K(\xi-\eta)\ip\eta^{2\beta}|\widehat g(\eta)|^2\dd\eta\\
 &\le C_p\int(1+|\xi-\eta|)^{-p}
                  \ip\eta^{2\beta}|\widehat g(\eta)|^2\dd\eta
 \end{aligned}
\]
for any fixed $p>d$, by the rapid decay of $\widehat\chi$.
Since a ball of radius $2^k$ can be covered by $C_d2^{kd}$ unit balls,
decomposition into a unit ball and dyadic annuli gives
\[
 \sup_\eta\int(1+|\xi-\eta|)^{-p}\dd\mu(\xi) \le
 C_dD_\mu\left(1+\sum_{k\ge1}2^{kd}2^{-p(k-1)}\right) \le C_{d,p}D_\mu.
\]
Integrating the pointwise bound against $\mu$ and applying Tonelli's
theorem, we obtain
\eqref{eq:fourier-measure-bound} from the Fourier definition of the
Sobolev norm.
\end{proof}

\begin{proof}[Proof of \cref{lem:sampled-energy}]
The sampled energy is
$\int_{\R^d}\ip\xi^{2s}|\widehat f(\xi)|^2\dd\Lambda_{M,q_M}(\xi)$,
where
\[
 \int F\dd\Lambda_{M,q_M}
 =\sum_iw_i\int_{-M}^M|\tau|^{d-1}F(\tau\omega_i)\dd\tau
\]
for measurable $F\ge0$. Since $\supp f\subset B_{R_0}$,
\eqref{eq:fourier-measure-bound} bounds this energy by
$C D_{\Lambda_{M,q_M}}\norm f_{H^s(\R^d)}^2$. The unit-ball masses satisfy
\begin{equation}\label{eq:frequency-unit-ball}
 \sup_{\xi\in\R^d}\Lambda_{M,q_M}(B(\xi,1))
 \le C\left[1+\left(\frac M{q_M}\right)^{d-1}\right].
\end{equation}
If $|\xi|\le2$, only $|\tau|\le3$ can contribute, and the fixed total
angular weight gives the bound. If $\rho=|\xi|>2$, the condition
$|\tau\omega_i-\xi|\le1$ implies
\[
 \big||\tau|-\rho\big|\le1,\qquad
 \left|\operatorname{sgn}(\tau)\omega_i-\frac\xi\rho\right|
 \le\frac2{|\tau|}\le\frac4\rho.
\]
The positive and negative radial intervals each have length at most
two, and their directions lie in two caps of radius $C/\rho$.
The spherical cap bound \eqref{eq:cap-mass} therefore bounds the mass
of the ball by
\[
 C\rho^{d-1}(\rho^{-1}+q_M^{-1})^{d-1}
 \le C[1+(\rho/q_M)^{d-1}].
\]
A ball intersecting the support of $\Lambda_{M,q_M}$ has
$\rho\le M+1$, which proves \eqref{eq:frequency-unit-ball}.
Substituting \eqref{eq:frequency-unit-ball} into the energy bound gives
the first inequality in \eqref{eq:sampled-energy}.
The second follows from $q_M\asymp M$ and
\eqref{eq:extension}.
\end{proof}

\begin{proof}[Proof of \cref{lem:shifted-synthesis}]
We first prove the estimate for $m=0$.
Set $\rho=\lambda-M\ge M$ and define
\[
 \int F\dd\mu_\lambda
 =\sum_iw_i\int_{-M}^M F((\lambda+\tau)\omega_i)\dd\tau.
\]
Every radial frequency $r=\lambda+\tau$ is at least $\rho$.
If $|r\omega_i-\xi|\le1$, then $|\xi|\ge\rho-1>0$ and
\[
 \left|\omega_i-\frac\xi{|\xi|}\right|
 \le\frac{|r\omega_i-\xi|+\big||\xi|-r\big|}{r}
 \le\frac2\rho.
\]
Each ray contributes a radial interval of length at most two, so
the spherical cap bound \eqref{eq:cap-mass} gives
\begin{equation*}
 \sup_{\xi\in\R^d}\mu_\lambda(B(\xi,1)) \le C(\rho^{-1}+q_M^{-1})^{d-1} \le
 C[\rho^{-(d-1)}+q_M^{-(d-1)}].
\end{equation*}

For $\psi\in L^2(\Omega)$, let $\psi_0$ be its zero extension.
Then $\supp\psi_0\subset\overline{B_{R_\Omega}}$ and
$\norm{\psi_0}_{L^2(\R^d)}=\norm{\psi}_{L^2(\Omega)}$.
By weighted Cauchy--Schwarz and \cref{lem:fourier-measure} with
$\beta=0$ and $A=R_\Omega+1$, the preceding mass bound gives
\[
 \begin{aligned}
 |\langle \mathcal T_\lambda z,\psi\rangle_{L^2(\Omega)}|
 &\le
 \left(\sum_iw_i\int_{-M}^M|z_i(\tau)|^2\dd\tau\right)^{1/2}
 \left(\int|\widehat{\psi_0}(\xi)|^2\dd\mu_\lambda(\xi)\right)^{1/2}\\
 &\le C[\rho^{-(d-1)/2}+q_M^{-(d-1)/2}]
 \left(\sum_iw_i\int_{-M}^M|z_i(\tau)|^2\dd\tau\right)^{1/2}
 \norm \psi_{L^2(\Omega)}.
 \end{aligned}
\]
Duality proves \eqref{eq:shifted-synthesis} for $m=0$.
For the $H^m$ estimate, differentiation with $|\alpha|\le m$
replaces each $z_i(\tau)$ by
\[
 (\ii(\lambda+\tau))^{|\alpha|}\omega_i^\alpha z_i(\tau).
\]
The weighted $L^2$ norm of these coefficients is at most
$(\lambda+M)^{|\alpha|}$ times that of $z$. Applying the $L^2$
estimate to each derivative and summing over $|\alpha|\le m$
completes the proof.
\end{proof}

\section{Auxiliary estimates for support truncation}
\label{app:support-truncation}
We prove \cref{lem:support-kernel,lem:tanh-remainder} using the
notation of Section~\ref{subsec:support-truncation}.

\begin{proof}[Proof of \cref{lem:support-kernel}]
We first construct a Schwartz function that agrees with $e^{-t}$ for
$t\ge0$ and whose Fourier transform vanishes to order at least $d-1$
at the origin.
Choose a fixed real smooth function $\chi$ such that
$\chi(t)=0$ for $t\le-1$ and $\chi(t)=1$ for $t\ge0$. Define
\begin{equation*}
 \Psi_0(t)=\chi(t)e^{-t},\qquad
 \Psi(t)=(-1)^{d-1}\Psi_0^{(d-1)}(t),\qquad k(r)=\widehat\Psi(r).
\end{equation*}
The functions $\Psi_0$, $\Psi$, and $k$ belong to the Schwartz class, and
\begin{equation}\label{eq:support-fourier-profile}
 \Psi(t)=e^{-t}\quad(t\ge0),\quad
 k(r)=(-\ii r)^{d-1}\widehat\Psi_0(r),\quad
 |k(r)|\le C_d\min\{|r|^{d-1},1\}.
\end{equation}
In particular, $k\in L^1(\R)$.
The zero of $k$ at the origin gives the factor $\delta^d$ in the
local mass estimate below.
Since $g_i(x)\ge0$, Fourier inversion gives
\begin{equation}\label{eq:support-fourier-representation}
 e^{-g_i(x)/\delta}
 =\frac{\delta}{2\pi}\int_\R
 k(\delta r)e^{\ii rH_\Omega(\omega_i)}e^{-\ii r\omega_i\cdot x}\dd r.
\end{equation}
Define a positive measure on frequency space by
\begin{equation*}
 \int_{\R^d}F(\xi)\dd\mu_{\delta,q_M}(\xi)
 =\sum_iw_i\int_\R\delta|k(\delta r)|F(-r\omega_i)\dd r
 \qquad(F\ge0).
\end{equation*}
Its total mass is $|\Sph|\norm k_{L^1}$. We claim that, for
$0<\delta\le1$,
\begin{equation}\label{eq:support-local-mass}
 D_{\mu_{\delta,q_M}}:=\sup_{\zeta\in\R^d}
       \mu_{\delta,q_M}(B(\zeta,1))
 \le C_d\bigl(\delta^d+\delta q_M^{-(d-1)}\bigr).
\end{equation}
If $|\zeta|\le2$, the relevant radial parameters satisfy $|r|\le3$.
Thus \eqref{eq:support-fourier-profile} and the total spherical weight
bound the mass by $C_d\delta^d$. If $\rho=|\zeta|>2$, then
$||r|-\rho|\le1$, and the corresponding directions lie in two caps
of radius $C/\rho$ about $\pm\zeta/\rho$. Each radial interval has
length at most two. The cap estimate \eqref{eq:cap-mass} therefore gives
\[
 \mu_{\delta,q_M}(B(\zeta,1)) \le C_d\delta\min\{(\delta\rho)^{d-1},1\}
 (\rho^{-1}+q_M^{-1})^{d-1} \le C_d\bigl(\delta^d+\delta q_M^{-(d-1)}\bigr),
\]
where the two terms follow respectively from
\[
 \delta\min\{(\delta\rho)^{d-1},1\}\rho^{-(d-1)}\le\delta^d,\qquad
 \delta\min\{(\delta\rho)^{d-1},1\}q_M^{-(d-1)} \le\delta q_M^{-(d-1)}.
\]
This proves \eqref{eq:support-local-mass}.

Let $S_\delta v=\sum_iw_iv_i e^{-g_i/\delta}$ and let $\psi_0$ be the
zero extension of $\psi\in L^2(\Omega)$. Then
$\supp\psi_0\subset B_{R_\Omega+1}$ and $\psi_0\in L^1\cap L^2$.
Since $k,\psi_0\in L^1$ and the direction set is finite, Fubini's
theorem and \eqref{eq:support-fourier-representation} give
\[
 \int_\Omega S_\delta v\,\overline{\psi}\dd x
 =\frac1{2\pi}\sum_iw_iv_i\int_\R
  \delta k(\delta r)e^{\ii rH_\Omega(\omega_i)}
       \overline{\widehat{\psi_0}(-r\omega_i)}\dd r.
\]
Applying Cauchy--Schwarz with weights $w_i$ in the sum over $i$
and weight $\delta|k(\delta r)|$ in the integral over $r$ yields
\begin{align*}
 \left|\int_\Omega S_\delta v\,\overline{\psi}\dd x\right|
 &\le\frac{\norm k_{L^1}^{1/2}}{2\pi}
       \left(\sum_iw_i|v_i|^2\right)^{1/2}
       \left(\int_{\R^d}|\widehat{\psi_0}(\xi)|^2
                          \dd\mu_{\delta,q_M}(\xi)\right)^{1/2}\\
 &\le C_{d,R_\Omega}D_{\mu_{\delta,q_M}}^{1/2}
       \left(\sum_iw_i|v_i|^2\right)^{1/2}\norm \psi_{L^2(\Omega)}.
\end{align*}
The second inequality follows from \cref{lem:fourier-measure} with
$\beta=0$ and $A=R_\Omega+1$. Duality and
\eqref{eq:support-local-mass} prove the result for $\delta\le1$.
For $\delta\ge1$, use $e^{-g_i/\delta}\le1$ and
\[
 \norm{S_\delta v}_{L^2(\Omega)}
 \le |\Omega|^{1/2}|\Sph|^{1/2}
       \left(\sum_iw_i|v_i|^2\right)^{1/2}.
\]
The factor on the right of \eqref{eq:support-kernel} is at least one,
so this also gives the stated bound.
\end{proof}

\begin{proof}[Proof of \cref{lem:tanh-remainder}]
Put $y=e^{-2z}$ and $B=2\cosh2$. Direct calculation gives
\[
 \sigma_D(z)=\frac{c_1y}{1+By+y^2},\qquad
 R_1(z)=-\frac{c_1y^2(B+y)}{1+By+y^2}.
\]
For $z\ge0$, write $R_1=y^2r(y)$ with $r\in C^\infty([0,1])$.
The identity $\dd/\dd z=-2y\,\dd/\dd y$ shows that every derivative
retains the factor $y^2$. For $z<0$, boundedness of the derivatives
of $\sigma_D$ gives $|R_1^{(k)}(z)|\le C_k(1+y)\le2C_k y^2$.
This proves \eqref{eq:tanh-remainder}. The leading exponential cancels
in $K_k$, and hence
\[
 K_k(z,v)=R_1^{(k)}(z+v)-e^{-2v}R_1^{(k)}(z).
\]
The bound \eqref{eq:tanh-remainder} and $e^{-4v}\le e^{-2v}$ give
the stated estimate.
\end{proof}

\bibliographystyle{amsplain}
\raggedbottom
\begingroup
\interlinepenalty=10000
\Needspace{5\baselineskip}
\bibliography{references}
\endgroup
\end{document}